\documentclass{amsart}%
\usepackage{amsfonts}
\usepackage{amsmath}
\usepackage{amssymb}
\usepackage{graphicx}%
\providecommand{\U}[1]{\protect\rule{.1in}{.1in}}
\newtheorem{theorem}{Theorem}
\theoremstyle{plain}

\newtheorem{axiom}{Assumption}

\newtheorem{lemma}{Lemma}

\theoremstyle{definition}
\newtheorem*{acknowledgement}{Acknowledgement}
\newtheorem{definition}{Definition}
\newtheorem{example}{Example}

\theoremstyle{remark}

\newtheorem{remark}{Remark}

\numberwithin{equation}{section}
\begin{document}
\title[Stability of locally $\mathbb{R}^{N}$-invariant solutions of Ricci flow]{Convergence and stability of locally $\mathbb{R}^{N}$-invariant solutions of
Ricci flow}
\author{Dan Knopf}
\address[Dan Knopf]{ University of Texas at Austin}
\email{danknopf@math.utexas.edu}
\urladdr{http://www.ma.utexas.edu/users/danknopf/}
\thanks{The author acknowledges NSF support in the form of grants DMS-0545984 and DMS-0505920.}

\begin{abstract}
Valuable models for immortal solutions of Ricci flow that collapse with
bounded curvature come from locally $\mathcal{G}$-invariant solutions on
bundles $\mathcal{G}^{N}\hookrightarrow\mathcal{M}\,\overset{\pi
}{\mathcal{\longrightarrow}}\,\mathcal{B}^{n}$, with $\mathcal{G}$ a nilpotent
Lie group. In this paper, we establish convergence and asymptotic stability,
modulo smooth finite-dimensional center manifolds, of certain $\mathbb{R}^{N}%
$-invariant model solutions. In case $N+n=3$, our results are relevant to work
of Lott classifying the asymptotic behavior of all $3$-dimensional Ricci flow
solutions whose sectional curvatures and diameters are respectively
$\mathcal{O}(t^{-1})$ and $\mathcal{O}(t^{1/2})$ as $t\rightarrow\infty$.

\end{abstract}


\maketitle
\tableofcontents

\section{Introduction}

There are many interesting open questions regarding the geometric and analytic
properties of immortal solutions $(\mathcal{M},g(t))$ of Ricci flow that
collapse. One does not expect such solutions to converge smoothly in a naive
sense. Instead, one may expect $(\mathcal{M},g(t))$ to exhibit
Gromov--Hausdorff convergence to a lower-dimensional geometric object. (See
\cite{Fukaya88}, for example.) An impelling motivation for studying such
solutions when the dimension of the total space is three is to obtain a better
understanding of the role of smooth collapse in the Ricci flow approach to
geometrization. See \cite{Lott07a} and \cite{GIK06} for a broader discussion
of some geometric and analytic problems related to collapse and their motivations.

In a recent paper \cite{Lott07}, John Lott makes significant progress in
understanding the long-time behavior of Ricci flow on a compact $3$-manifold.
He proves that any solution $(\mathcal{M}^{3},g(t))$ of Ricci flow that exists
for all $t\geq0$ with sectional curvatures that are $\mathcal{O}(t^{-1})$ and
diameter that is $\mathcal{O}(t^{1/2})$ as $t\rightarrow\infty$ converges
(after pullback to the universal cover and modification by diffeomorphisms) to
an expanding homogeneous soliton. A key component in the proof is the analysis
of locally $\mathcal{G}$-invariant solutions of Ricci flow, where
$\mathcal{G}$ is a connected $N$-dimensional Abelian Lie group. Such solutions
are expected to constitute valuable models for the behavior of Type-III
(immortal) solutions of Ricci flow in all dimensions, especially solutions
undergoing collapse. We now review the basic setup from \cite{Lott07}.

Let $\mathbb{R}^{N}\hookrightarrow\mathcal{M}\,\overset{\pi}%
{\mathcal{\longrightarrow}}\,\mathcal{B}^{n}$ be a fiber bundle over a
connected oriented compact base $\mathcal{B}$, and let $\mathbb{R}%
^{N}\hookrightarrow\mathcal{E}\,\overset{p}{\mathcal{\longrightarrow}%
}\,\mathcal{B}^{n}$ be a flat vector bundle over $\mathcal{B}$. Let
$\mathcal{E}\times_{\mathcal{B}}\mathcal{M}$ denote the fiber space $%
{\textstyle\bigcup\nolimits_{b\in\mathcal{B}}}
(\mathcal{E}_{b}\times\mathcal{M}_{b})$. Assume that there exists a smooth map
$\mathcal{E}\times_{\mathcal{B}}\mathcal{M}\rightarrow\mathcal{M}$ such that
over each basepoint $b\in\mathcal{B}$, the map $\mathcal{E}_{b}\times
\mathcal{M}_{b}\rightarrow\mathcal{M}_{b}$ is a free transitive action. Assume
also that this action is consistent with the flat connection on $\mathcal{E}$
in the sense that for any open set $\mathcal{U}\subseteq\mathcal{B}$ small
enough that $\mathcal{E}|_{\mathcal{U}}\approx\mathcal{U}\times\mathbb{R}^{N}$
is a local trivialization, the set $\pi^{-1}(\mathcal{U})$ has a free
$\mathbb{R}^{N}$ action, hence is the total space of a principal
$\mathbb{R}^{N}$ bundle over $\mathcal{U}$. In this way, $\mathcal{M}$ may be
regarded as an $\mathbb{R}^{N}$-principal bundle over $\mathcal{B}$ twisted by
the flat vector bundle $\mathcal{E}$. One can define a connection $A$ on
$\mathcal{M}$ with the property that any restriction $A|_{\pi^{-1}%
(\mathcal{U})}$ is an $\mathbb{R}^{N}$-valued connection.

Let $\mathcal{U}\subseteq\mathcal{B}$ be an open set small enough such that
$\mathcal{E}|_{\mathcal{U}}\approx\mathcal{U}\times\mathbb{R}^{N}$ is a local
trivialization and such that there exist a local parameterization
$\rho:\mathbb{R}^{n}\rightarrow\mathcal{U}$ and a local section $\sigma
:\mathcal{U}\rightarrow\pi^{-1}(\mathcal{U})$. Then any choice of basis
$(e_{1},\ldots,e_{N})$ for $\mathbb{R}^{N}$ yields local coordinates on
$\pi^{-1}(\mathcal{U})$ via%
\[
\mathbb{R}^{n}\times\mathbb{R}^{N}\ni(x^{\alpha},x^{i})\mapsto(x^{i}%
e_{i})\cdot\sigma(\rho(x^{\alpha})),
\]
where $\{x^{\alpha}\}_{\alpha=1}^{n}$are coordinates on $\mathcal{U}$,
$\{x^{i}\}_{i=1}^{N}$are coordinates on $\mathbb{R}^{N}$, and $\cdot$ denotes
the free $\mathbb{R}^{N}$-action described above.

Let $\mathbf{\bar{g}}$ denote a Riemannian metric on $\mathcal{M}$ with the
property that this action is a local isometry. Then with respect to these
coordinates, one may write%
\begin{equation}
\mathbf{\bar{g}}=\sum_{\alpha,\beta=1}^{n}\bar{g}_{\alpha\beta}\,dx^{\alpha
}\,dx^{\beta}+\sum_{i,j=1}^{N}\bar{G}_{ij}(dx^{i}+%
{\textstyle\sum_{\alpha=1}^{n}}
\bar{A}_{\alpha}^{i}\,dx^{\alpha})(dx^{j}+%
{\textstyle\sum_{\beta=1}^{n}}
\bar{A}_{\beta}^{j}\,dx^{\beta}). \label{g-bar}%
\end{equation}
Observe here that for $b\in\mathcal{U}\subseteq\mathcal{B}$, $\sum
_{\alpha,\beta=1}^{n}\bar{g}_{\alpha\beta}(b)\,dx^{\alpha}\,dx^{\beta}$ is the
local expression of a Riemannian metric on $\mathcal{B}$, $\bar{A}_{\alpha
}^{i}(b)\,dx^{\alpha}$ is locally the pullback by $\sigma$ of a connection on
$\pi^{-1}(\mathcal{U})\rightarrow\mathcal{U}$, and $\sum_{i,j=1}^{N}\bar
{G}_{ij}(b)\,dx^{i}\,dx^{j}$ gives a Euclidean inner product on the fiber over
$b$.

A one-parameter family $(\mathcal{M},\mathbf{\bar{g}}(t):t\in\mathcal{I})$ of
such metrics evolving by Ricci flow in a nonempty time interval $\mathcal{I}$
constitutes a \emph{locally }$\mathbb{R}^{N}$\emph{-invariant Ricci flow
solution.} In \cite{Lott07}, Lott shows that such a solution is equivalent,
modulo diffeomorphisms of $\mathcal{M}$ and modifications of $\bar{A}$ by
exact forms, to the system\thinspace\footnote{Here $\bar{\nabla}_{\alpha
}=\frac{\partial}{\partial x^{\alpha}}$, $\bar{\nabla}^{\alpha}=\bar
{g}^{\alpha\beta}\frac{\partial}{\partial x^{\beta}}$, $(d\bar{A}%
)_{\alpha\beta}^{i}=\bar{\nabla}_{\alpha}\bar{A}_{\beta}^{i}-\bar{\nabla
}_{\beta}\bar{A}_{\alpha}^{i}$, $(\bar{\delta}d\bar{A})_{\alpha}^{i}%
=-\bar{\nabla}^{\beta}(d\bar{A})_{\beta\alpha}^{i}$, and $\bar{\Delta}\bar
{G}_{ij}=\bar{g}^{\alpha\beta}\bar{\nabla}_{\alpha}\bar{\nabla}_{\beta}\bar
{G}_{ij}=\bar{g}^{\alpha\beta}(\frac{\partial^{2}}{\partial x^{\alpha}\partial
x^{\beta}}\bar{G}_{ij}-\bar{\Gamma}_{\alpha\beta}^{\gamma}\frac{\partial
}{\partial x^{\gamma}}\bar{G}_{ij})$, where $\bar{\Gamma}$ represents the
Levi-Civita connection of $\bar{g}$.}
\begin{subequations}
\label{RFS}%
\begin{align}
\frac{\partial}{\partial t}\bar{g}_{\alpha\beta}  &  =-2\bar{R}_{\alpha\beta
}+\frac{1}{2}\bar{G}^{ik}\bar{G}^{j\ell}\bar{\nabla}_{\alpha}\bar{G}_{ij}%
\bar{\nabla}_{\beta}\bar{G}_{k\ell}+\bar{g}^{\gamma\delta}\bar{G}_{ij}%
(d\bar{A})_{\alpha\gamma}^{i}(d\bar{A})_{\beta\delta}^{j},\\
\frac{\partial}{\partial t}\bar{A}_{\alpha}^{i}  &  =-(\bar{\delta}d\bar
{A})_{\alpha}^{i}+\bar{G}^{ij}\bar{\nabla}^{\beta}\bar{G}_{jk}(d\bar
{A})_{\beta\alpha}^{k},\\
\frac{\partial}{\partial t}\bar{G}_{ij}  &  =\bar{\Delta}\bar{G}_{ij}-\bar
{G}^{k\ell}\bar{\nabla}_{\alpha}\bar{G}_{ik}\bar{\nabla}^{\alpha}\bar{G}_{\ell
j}-\frac{1}{2}\bar{g}^{\alpha\gamma}\bar{g}^{\beta\delta}\bar{G}_{ik}\bar
{G}_{j\ell}(d\bar{A})_{\alpha\beta}^{k}(d\bar{A})_{\gamma\delta}^{\ell}.
\end{align}

Abusing notation, we denote a solution of (\ref{RFS}) by $\mathbf{\bar{g}%
}(t)=(\bar{g}(t),\bar{A}(t),\bar{G}(t))$. In order to study the long-time
behavior of such systems, we engineer a transformation that turns certain
model solutions into fixed points whose asymptotic stability can be
investigated. To facilitate this, we assume that $\mathcal{M}$ admits a flat
connection, allowing us to regard $\bar{A}$ as an $\mathbb{R}^{N}$-valued
$1$-form. Let a function $s(t)$ and a constant $c$ be given. Let $\sigma(t)$
be any positive antiderivative of $s$. Consider the transformation
\end{subequations}
\[
(\bar{g}(t),\bar{A}(t),\bar{G}(t))=\mathbf{\bar{g}}(t)\mapsto\mathbf{g}%
(\tau)=(g(\tau),A(\tau),G(\tau)),
\]
where%
\[
g=\sigma^{-1}\bar{g},\qquad A=\sigma^{-\frac{1+c}{2}}\bar{A},\qquad
G=\sigma^{c}\bar{G},\qquad\tau=\int_{t_{0}}^{t}\sigma^{-1}(\bar{t})\,d\bar
{t},\qquad(t_{0}\in\mathcal{I}).
\]
Observe that the exponent $-(1+c)/2$ above is necessary so that no factors of
$\sigma$ appear in the transformed system~(\ref{RRFS}) below.
Examples~\ref{Product}--\ref{Sol} below illustrate why this is desirable.
Under the transformation $\mathbf{\bar{g}}(t)\mapsto\mathbf{g}(\tau)$,
system~(\ref{RFS}) becomes\thinspace\footnote{Here $\nabla_{\alpha}%
=\frac{\partial}{\partial x^{\alpha}}$, $\nabla^{\alpha}=g^{\alpha\beta}%
\frac{\partial}{\partial x^{\beta}}$, $(dA)_{\alpha\beta}^{i}=\nabla_{\alpha
}A_{\beta}^{i}-\nabla_{\beta}A_{\alpha}^{i}$, $(\delta dA)_{\alpha}%
^{i}=-\nabla^{\beta}(dA)_{\beta\alpha}^{i}$, and $\Delta G_{ij}=g^{\alpha
\beta}\nabla_{\alpha}\nabla_{\beta}G_{ij}=g^{\alpha\beta}(\frac{\partial^{2}%
}{\partial x^{\alpha}\partial x^{\beta}}G_{ij}-\Gamma_{\alpha\beta}^{\gamma
}\frac{\partial}{\partial x^{\gamma}}G_{ij})$, where $\Gamma$ represents the
Levi-Civita connection of $g$.}
\begin{subequations}
\label{RRFS}%
\begin{align}
\frac{\partial}{\partial\tau}g_{\alpha\beta}  &  =-2R_{\alpha\beta}+\frac
{1}{2}G^{ik}G^{j\ell}\nabla_{\alpha}G_{ij}\nabla_{\beta}G_{k\ell}%
+g^{\gamma\delta}G_{ij}(dA)_{\alpha\gamma}^{i}(dA)_{\beta\delta}%
^{j}-sg_{\alpha\beta},\\
\frac{\partial}{\partial\tau}A_{\alpha}^{i}  &  =-(\delta dA)_{\alpha}%
^{i}+G^{ij}\nabla^{\beta}G_{jk}(dA)_{\beta\alpha}^{k}-\frac{1+c}{2}sA_{\alpha
}^{i},\\
\frac{\partial}{\partial\tau}G_{ij}  &  =\Delta G_{ij}-G^{k\ell}\nabla
_{\alpha}G_{ik}\nabla^{\alpha}G_{\ell j}-\frac{1}{2}g^{\alpha\gamma}%
g^{\beta\delta}G_{ik}G_{j\ell}(dA)_{\alpha\beta}^{k}(dA)_{\gamma\delta}^{\ell
}+csG_{ij}.
\end{align}
We call this system a \emph{rescaled locally }$\mathbb{R}^{N}$\emph{-invariant
Ricci flow.}
\end{subequations}
\begin{example}
\label{Product}On a solution $(\mathbb{R}^{N}\times\mathcal{B},\mathbf{\bar
{g}}(t))$ of (\ref{RFS}) that is a Riemannian product over a nonflat Einstein
base $(\mathcal{B},\bar{g})$, one may choose coordinates so that $\bar{G}$ is
constant, $\bar{A}$ vanishes, and $\bar{g}(t)=-\varkappa t\,\bar{g}%
(-\varkappa)$, where $\varkappa=\pm1$ is the Einstein constant such that
$2\operatorname*{Rc}(\bar{g}(-\varkappa))=\varkappa\bar{g}(-\varkappa)$. (Note
that $\bar{g}(t)$ exists for $t<0$ if $\varkappa=1$ and for $t>0$ if
$\varkappa=-1$.) The choices $s=-\varkappa$, $c=0$, and $t_{0}=-\varkappa$
transform $\mathbf{\bar{g}}(t_{0})$ into a stationary solution $\mathbf{g}(0)$
of the autonomous system (\ref{RRFS}).
\end{example}

\begin{example}
\label{NRF}A somewhat more general normalization is as follows. Suppose that
$g(\tau)$ may be regarded as a well-defined metric on $\mathcal{B}$. Let
$V(\tau)$ denote the volume of $(\mathcal{B},g(\tau))$ and define%
\[
r=R-\frac{1}{4}\left\vert \nabla G\right\vert ^{2}-\frac{1}{2}\left\vert
dA\right\vert ^{2},
\]
where everything is computed with respect to $g$. Because $\frac{dV}{d\tau
}=-\int_{\mathcal{B}}r\,\mathrm{d\mu}\,-\frac{n}{2}sV(\tau)$, it follows that
$V$ is fixed if and only if\thinspace\footnote{Given any smooth function
$f:\mathcal{B}\rightarrow\mathbb{R}$, we define $\oint_{\mathcal{B}}%
f\,d\mu=\frac{\int_{\mathcal{B}}f\,d\mu}{\int_{\mathcal{B}}d\mu}$.}%
\[
s=-\frac{2}{n}\oint_{\mathcal{B}}r\,\mathrm{d\mu}\,.
\]
Consider a Riemannian product solution over an Einstein base. Such a solution
may be written as $\mathbf{\bar{g}}(t)=(\bar{g}(t),\bar{A},\bar{G})$, where
$\bar{A}$ vanishes identically and $\bar{G}$ is fixed in space and time. For
any $t_{0}$ in its time domain of existence, taking $c=0$ makes $\mathbf{g}%
(0)=(\sigma^{-1}(t_{0})\bar{g}(t_{0}),\,0,\,\bar{G})$ into a stationary
solution of (\ref{RRFS}) for any choice of $\sigma^{-1}(t_{0})>0$.
\end{example}

\begin{example}
\label{Sol}The mapping torus $\mathcal{M}_{f}$ of $f:\mathbb{R}^{N}%
\rightarrow\mathbb{R}^{N}$ is $([0,1]\times\mathbb{R}^{N})/\sim$, where
$(0,p)\sim(1,f(p))$. With respect to coordinates $x^{0}\in\mathcal{S}^{1}$ and
$(x^{1},\ldots,x^{N})\in\mathbb{R}^{N}$ on the circle bundle $\mathbb{R}%
^{N}\hookrightarrow\mathcal{M}_{f}\,\overset{\pi}{\mathcal{\longrightarrow}%
}\,\mathcal{S}^{1}$, a locally $\mathbb{R}^{N}$-invariant metric has the form%
\[
\mathbf{g}=u\,dx^{0}\otimes dx^{0}+A_{i}\,(dx^{0}\otimes dx^{i}+dx^{i}\otimes
dx^{0})+G_{ij}\,dx^{i}\otimes dx^{j},
\]
where $u$, $A^{i}$, and $G$ depend only on $x^{0}$, and $A_{i}=G_{ij}A^{j}$.
System~(\ref{RRFS}) reduces to
\begin{subequations}
\label{n=1}%
\begin{align}
\frac{\partial}{\partial\tau}u  &  =\frac{1}{2}\left\vert \partial
_{0}G\right\vert ^{2}-su,\label{n=1g}\\
\frac{\partial}{\partial\tau}A^{i}  &  =-\frac{1+c}{2}sA^{i},\label{n=1A}\\
\frac{\partial}{\partial\tau}G_{ij}  &  =u^{-1}(\nabla_{0}^{2}G_{ij}-G^{k\ell
}\partial_{0}G_{ik}\partial_{0}G_{\ell j})+csG_{ij}, \label{n=1G}%
\end{align}
where $\left\vert \partial_{0}G\right\vert ^{2}=G^{ik}G^{j\ell}\partial
_{0}G_{ij}\partial_{0}G_{k\ell}$ and $\nabla_{0}^{2}=\partial_{0}%
^{2}-\mathbf{\Gamma}_{00}^{0}\partial_{0}$. A motivating case occurs when
$N=2$. Compact $3$-manifolds with $\operatorname*{sol}$ geometry are mapping
tori of automorphisms $\mathcal{T}^{2}\rightarrow\mathcal{T}^{2}$ induced by
$\Lambda\in\operatorname*{SL}(2,\mathbb{Z})$ with eigenvalues $\lambda
^{-1}<1<\lambda$. Now let $\mathbf{g}$ be a locally homogeneous
$\operatorname*{sol}$-Gowdy metric of the kind considered as initial data in
\cite{HI93} and \cite{Knopf00}, so that $A=0$. Without loss of generality, one
may parameterize the initial data by arc length (thereby setting $u=1$ at
$t=0$) and thereby write $\mathbf{g}=dx^{0}\otimes dx^{0}+G_{ij}%
\,dx^{i}\,dx^{j}$, where one identifies $G$ with the matrix%
\end{subequations}
\[%
\begin{pmatrix}
e^{F+W} & 0\\
0 & e^{F-W}%
\end{pmatrix}
.
\]
Here $F$ is an arbitrary constant and $\partial_{0}W=2\log\lambda$ is a
topological constant determined by the holonomy $\Lambda$ around
$\mathcal{S}^{1}=[0,1]/\sim$. It is easy to check that this $\mathbf{g}$
becomes a stationary solution of (\ref{n=1}) if and only if $s=(2\log
\lambda)^{2}>0$ and $c=0$. For the corresponding unscaled solution
$\mathbf{\bar{g}}=(\bar{u}(t),\bar{A}(t),\bar{G}(t))$, one has $\bar{u}=1+st$,
with $\bar{A}=0$ and $\bar{G}$ fixed in time.
\end{example}

Our purpose in this note is to investigate the stability of certain natural
fixed points of system~(\ref{RRFS}) like those in Examples~(\ref{Product}%
)--(\ref{Sol}). More precisely, we establish convergence and asymptotic
stability of rescaled Ricci flow, modulo smooth finite-dimensional center
manifolds, for all initial data sufficiently near certain model solutions in
the cases $N=1$ or $n=1$. These are the cases most amenable to study by
linearization and also most relevant to applications when the total dimension
is $N+n=3$, e.g.~to \cite{Lott07}. Because stability of flat metrics (modulo
smooth finite-dimensional center manifolds) was established in \cite{GIK02},
we assume in what follows that the model solutions in question are not flat.

A standard method for establishing stability of a nonlinear flow
$\widetilde{\mathbf{g}}_{\tau}\mathbf{=F}(\widetilde{\mathbf{g}})$ near a
stationary solution $\mathbf{g}$ is to proceed in two steps. (1) Compute the
linearized operator $\mathbf{L}_{\mathbf{g}}^{\mathbb{C}}$ and establish that
it is sectorial with stable spectrum.\footnote{See
Definition~\ref{DefineSectorial} on page~\pageref{DefineSectorial}.} (2)
Deduce stability of the nonlinear flow in appropriate function spaces from
properties of its linearization. For Ricci flow, this method is often not as
easy as one might expect. There are several reasons for this. To apply
existing theory such as \cite{DL85}, one requires the linearization to be
elliptic. But because Ricci flow is invariant under the full
infinite-dimensional diffeomorphism group, its linearization $\mathbf{L}%
_{\mathbf{g}}^{\mathbb{C}}$ is not elliptic unless one fixes a gauge,
typically by the introduction of DeTurck diffeomorphisms. Even then,
$\mathbf{L}_{\mathbf{g}}^{\mathbb{C}}$ may fail to be self adjoint. (For
example, see \cite{GIK06} and Section~\ref{MTRF} below.) Furthermore, in many
cases of interest, such as those considered in \cite{GIK02} and in this paper,
the spectrum of $\mathbf{L}_{\mathbf{g}}^{\mathbb{C}}$ has nonempty
intersection with the imaginary axis. In general, this allows families of
stationary or slowly changing solutions to compose local $C^{r}$ center
manifolds, which can shrink as $r$ increases.\footnote{For example, the
\textquotedblleft inner layer\textquotedblright\ asymptotics of a Ricci flow
neckpinch correspond to non-stationary solutions in the kernel of its
linearization at the cylinder soliton. See \cite{AK2}.} Finally, the DeTurck
diffeomorphisms can introduce instabilities that may not be present in the
original equation. (For example, the DeTurck diffeomorphisms solve a harmonic
map flow. But the identity map $\operatorname*{id}:\mathcal{S}^{n}%
\rightarrow\mathcal{S}^{n}$ of the round sphere is unstable under harmonic map
flow for all $n\geq3$. See Remark~\ref{UnstableHMF} below.) In spite of these
obstacles, stability results are obtained in \cite{GIK02} and (by a somewhat
different method) by \v{S}e\v{s}um in \cite{Sesum06}.

In this paper, we prove the following results that imply convergence in
little-H\"{o}lder spaces to be defined in Section~\ref{MRS}, below:

\begin{theorem}
\label{HyperbolicIsStable}Let $\mathbf{g}=(g,A,u)$ be a locally $\mathbb{R}%
^{1}$-invariant metric of the form (\ref{KaluzaKleinMetric}) on a product
$\mathbb{R}^{1}\times\mathcal{B}$, where $\mathcal{B}$ is compact and
orientable. Suppose that $g$ has constant sectional curvature $-1/2(n-1)$, $A$
vanishes, and $u$ is constant. Then for any $\rho\in(0,1)$, there exists
$\theta\in(\rho,1)$ such that the following holds.

There exists a $(1+\theta)$ little-H\"{o}lder neighborhood $\mathcal{U}$ of
$\mathbf{g}$ such that for all initial data $\widetilde{\mathbf{g}}%
(0)\in\mathcal{U}$, the unique solution $\widetilde{\mathbf{g}}(\tau)$ of
$\varkappa$-rescaled locally $\mathbb{R}^{1}$-invariant Ricci flow
(\ref{kappa-rescaled}) exists for all $\tau\geq0$ and converges exponentially
fast in the $(2+\rho)$-H\"{o}lder norm to a limit metric $\mathbf{g}_{\infty
}=(g_{\infty},A_{\infty},u_{\infty})$ such that $g_{\infty}$ is hyperbolic,
$A_{\infty}$ vanishes, and $u_{\infty}$ is constant.
\end{theorem}

\begin{theorem}
\label{TwoSphereIsStable}Let $\mathbf{g}=(g,A,u)$ be a locally $\mathbb{R}%
^{1}$-invariant metric of the form (\ref{KaluzaKleinMetric2}) on a product
$\mathbb{R}^{1}\times\mathcal{S}^{2}$. Suppose that $g$ has constant positive
sectional curvature, $A$ vanishes, and $u$ is constant. Then for any $\rho
\in(0,1)$, there exists $\theta\in(\rho,1)$ such that the following holds.

There exists a $(1+\theta)$ little-H\"{o}lder neighborhood $\mathcal{U}$ of
$\mathbf{g}$ such that for all initial data $\widetilde{\mathbf{g}}%
(0)\in\mathcal{U}$, the unique solution $\widetilde{\mathbf{g}}(\tau)$ of
volume-rescaled locally $\mathbb{R}^{1}$-invariant Ricci flow
(\ref{volume-rescaled}) exists for all $\tau\geq0$ and converges exponentially
fast in the $(2+\rho)$-H\"{o}lder norm to a limit metric $\mathbf{g}_{\infty
}=(g_{\infty},A_{\infty},u_{\infty})$ such that $g_{\infty}$ has constant
positive sectional curvature, $A_{\infty}$ vanishes, and $u_{\infty}$ is constant.
\end{theorem}

\begin{theorem}
\label{SolIsStable}Let $\mathbf{g}=(u,A,G)$ be a metric of the form
(\ref{CircleBundle}) on the mapping torus $\mathbb{R}^{N}\hookrightarrow
\mathcal{M}_{\Lambda}\overset{\pi}{\mathcal{\longrightarrow}}\mathcal{S}^{1}$
of a given nonzero $\Lambda\in\operatorname*{Gl}(N,\mathbb{R})$ admitting a
flat connection. Suppose that $u=1$, $A$ vanishes, and $G$ is a
harmonic-Einstein metric (\ref{Harmonic-Einstein}). Then for any $\rho
\in(0,1)$, there exists $\theta\in(\rho,1)$ such that the following holds.

There exists a $(1+\theta)$ little-H\"{o}lder neighborhood $\mathcal{U}$ of
$\mathbf{g}$ such that for all initial data $\widetilde{\mathbf{g}}%
(0)\in\mathcal{U}$, the unique solution $\widetilde{\mathbf{g}}(\tau)$ of
holonomy-rescaled locally $\mathbb{R}^{N}$-invariant Ricci flow
(\ref{MappingTorus}) exists for all $\tau\geq0$ and converges exponentially
fast in the $(2+\rho)$-H\"{o}lder norm to a limit metric $\mathbf{g}_{\infty
}=(u_{\infty},A_{\infty},G_{\infty})$ such that $u_{\infty}=1$, $A_{\infty}$
vanishes, and $G_{\infty}$ is harmonic-Einstein.
\end{theorem}

This paper is organized as follows. In Section~\ref{Theory}, we review a
general theory of asymptotic stability for quasilinear \textsc{pde} in the
presence of center manifolds, then establish the context in which that theory
applies here. In Sections~\ref{KappaFlow}, \ref{VolumeFlow}, and \ref{MTRF},
we study and prove stability of $\varkappa$-rescaled flows, volume-rescaled
flows, and holonomy-rescaled flows, respectively. Each of these is a
suitably-chosen variant of (\ref{RRFS}). In Remarks~\ref{kappa-bad} and
\ref{volume-bad}, below, we provide (partial) explanations of these apparently
\emph{ad hoc }choices of normalization.

\begin{acknowledgement}
I\ warmly thank John Lott for sharing his results with me and raising the
questions that are addressed in this note.
\end{acknowledgement}

\begin{acknowledgement}
I also thank James Isenberg for introducing me to several problems closely
related to those considered in this note. Motivated by Kaluza--Klein theories
in physics, Dedrickson and Isenberg did early work on locally $\mathbb{R}^{1}%
$-invariant solutions (both with and without a connection), as well as
generalizations of Example~\ref{Sol} to higher-dimensional fibers (without a connection).
\end{acknowledgement}

\begin{acknowledgement}
My sincere thanks go to the referees for their suggestions to improve the
exposition and notation of this paper.
\end{acknowledgement}

\section{Center manifold stability theory for quasilinear
systems\label{Theory}}

In this section, we recall relevant aspects of the theory that allows one to
derive rigorous conclusions about the asymptotic behavior of the nonlinear
system (\ref{RRFS}) from its linearization. Good sources are \cite{DL85} and
\cite{Simonett95}. Our approach here mainly follows the latter. The reader
familiar with \cite{GIK02} can safely skim this section.

\subsection{Maximal regularity spaces\label{MRS}}

Because optimal asymptotic stability results are obtained in continuous
interpolation spaces, we begin by recalling aspects of the \emph{maximal
regularity} theory\emph{ }of Da Prato and Grisvard \cite{DG79}.
(Section~\ref{AGCMT} below provides some motivation for our use of this theory.)

Let $\mathcal{Y}$ be a compact Riemannian manifold, possible with boundary
$\partial\mathcal{Y}$. The main cases we have in mind in this note are: (1)
$\mathcal{Y}$ is a compact hyperbolic manifold; (2) $\mathcal{Y}$ is the round
sphere; or (3) $\mathcal{Y}=[0,1]$.

Let $\Sigma_{0}(\mathcal{Y})$, $\Sigma_{1}(\mathcal{Y})$, and $\Sigma
_{2}(\mathcal{Y})$ respectively denote the spaces of $C^{\infty}$ functions,
$C^{\infty}$ $1$-forms, and $C^{\infty}$ symmetric $(2,0)$-tensor fields
supported on $\mathcal{Y}$. If $\partial\mathcal{Y}\neq\emptyset$, as in
case~(3), we restrict to those functions, forms, and tensor fields satisfying
prescribed linear boundary conditions like (\ref{MTboundary-v}%
)--(\ref{MTboundary-h}), respectively.

Given $r\in\mathbb{N}$ and $\rho\in(0,1)$, let $\Sigma_{0}^{r+\rho
}(\mathcal{Y})$, $\Sigma_{1}^{r+\rho}(\mathcal{Y})$, and $\Sigma_{2}^{r+\rho
}(\mathcal{Y})$ denote the closures of $\Sigma_{0}(\mathcal{Y})$, $\Sigma
_{1}(\mathcal{Y})$, and $\Sigma_{2}(\mathcal{Y})$, respectively, with respect
to the relevant $r+\rho$ H\"{o}lder norms. These are the well-known
\emph{little-H\"{o}lder spaces.} (Recall that $C^{\infty}$ representatives are
not dense in the usual H\"{o}lder spaces.) Any little-H\"{o}lder space
$\mathfrak{h}^{r+\rho}$ is a Banach space. Moreover, if $s\leq r$ and
$\sigma\leq\rho$, then there is a continuous and dense inclusion
$\mathfrak{h}^{r+\rho}\hookrightarrow\mathfrak{h}^{s+\sigma}$.

Given a continuous dense inclusion $B_{1}\hookrightarrow B_{0}$ of Banach
spaces and $\theta\in(0,1)$, the \emph{continuous interpolation space}
$(B_{0},B_{1})_{\theta}$ is the set of all $x\in B_{0}$ such that there exist
sequences $\{y_{n}\}\subset B_{0}$ and $\{z_{n}\}\subset B_{1}$ satisfying%
\begin{equation}
\left\{
\begin{array}
[c]{c}%
x=y_{n}+z_{n}\\
\mathstrut\\
\left\Vert y_{n}\right\Vert _{B_{0}}=o(2^{-n\theta})\text{ as }n\rightarrow
\infty\\
\mathstrut\\
\left\Vert z_{n}\right\Vert _{B_{1}}=o(2^{n(1-\theta)})\text{ as }%
n\rightarrow\infty.
\end{array}
\right.  \label{DefineCISpaces}%
\end{equation}
Continuous interpolation spaces are introduced in \cite{DG79}. By \cite{DF87},
they are equivalent in norm to the real interpolation spaces also found in the
literature. In fact, the norm $\left\Vert x\right\Vert _{\theta}$ on
$(B_{0},B_{1})_{\theta}$ is equivalent to%
\[
\inf\left\{  \sup_{n\in\mathbb{N}}\{2^{n\theta}\left\Vert y_{n}\right\Vert
_{B_{0}},\ 2^{-n(1-\theta)}\left\Vert z_{n}\right\Vert _{B_{1}}\}:\{y_{n}%
\}\text{ and}\ \{z_{n}\}\text{ satisfy (\ref{DefineCISpaces})}\right\}  .
\]

The little-H\"{o}lder spaces are particularly well adapted to continuous
interpolation. (See \cite{Triebel95}.) Indeed, let $\mathfrak{h}^{r+\rho}$ and
$\mathfrak{h}^{s+\sigma}$ be little-H\"{o}lder spaces. Let $s\leq
r\in\mathbb{N}$, $0<\sigma<\rho<1$, and $0<\theta<1$. If $\theta
(r+\rho)+(1-\theta)(s+\sigma)\notin\mathbb{N}$, then there is a Banach space
isomorphism%
\begin{equation}
(\mathfrak{h}^{s+\sigma},\mathfrak{h}^{r+\rho})_{\theta}\cong\mathfrak{h}%
^{[\theta r+(1-\theta)s]+[\theta\rho+(1-\theta)\sigma]}
\label{InterpolationIsomorphism}%
\end{equation}
and there exists $C<\infty$ such that for all $\eta\in\mathfrak{h}^{r+\rho}$,
one has%
\begin{equation}
\left\Vert \eta\right\Vert _{(\mathfrak{h}^{s+\sigma},\mathfrak{h}^{r+\rho
})_{\theta}}\leq C\left\Vert \eta\right\Vert _{\mathfrak{h}^{s+\sigma}%
}^{1-\theta}\left\Vert \eta\right\Vert _{\mathfrak{h}^{r+\rho}}^{\theta}.
\label{InterpolationNorm}%
\end{equation}
These properties make the little-H\"{o}lder spaces highly useful for our purposes.

\subsection{A general center manifold theorem\label{AGCMT}}

To prove the results in this note, we invoke a special case of a theorem of
Simonett \cite{Simonett95}, whose hypotheses we now recall. (Compare
\cite{GIK02}.) Our notation is as follows. If $\mathbb{X}$ and $\mathbb{Y}$
are Banach spaces, then $\mathcal{L}(\mathbb{X},\mathbb{Y})$ is the set of
bounded linear maps $\mathbb{X}\rightarrow\mathbb{Y}$. If $(\mathbb{X},d)$ is
a metric space, then $B(\mathbb{X},x,r)$ is the open ball of radius $r>0$
centered at $x\in\mathbb{X}$. If $\mathbf{L}$ is a linear operator on a real
space, we denote its natural complexification by $\mathbf{L}^{\mathbb{C}%
}(u+iv)=\mathbf{L}u+i\mathbf{L}v$.

If $\mathbb{X}\hookrightarrow\mathbb{Y}$ is a continuous and dense inclusion
of Banach spaces and $\mathbf{Q}:\mathbb{X}\rightarrow\mathbb{Y}$ is a
nonlinear differential operator, $\mathbb{X}\ni\mathbf{g}\mapsto
\mathbf{Q}(\mathbf{g})\in\mathbb{Y}$, we denote its linearization at
$\widehat{\mathbf{g}}\in\mathbb{X}$ by $\mathbf{L}_{\widehat{\mathbf{g}}%
}=\mathbf{Q}^{\prime}(\widehat{\mathbf{g}}):\mathbb{D}(\mathbf{L}%
_{\widehat{\mathbf{g}}})\subseteq\mathbb{Y}\rightarrow\mathbb{Y}$. Our main
assumptions below are that $\mathbf{g}\mapsto\mathbf{Q}(\mathbf{g})$ is a
quasilinear differential operator and that $\mathbf{L}_{\widehat{\mathbf{g}}%
}:\mathbf{g}\mapsto\mathbf{L}_{\widehat{\mathbf{g}}}(\mathbf{g})$ generates an
analytic strongly-continuous semigroup on $\mathcal{L}(\mathbb{Y},\mathbb{Y})$.

To apply Simonett's stability theorem, one must verify certain technical
hypotheses with respect to a given quasilinear parabolic equation and ordered
Banach spaces satisfying%
\begin{equation}
\mathbb{X}_{1}\subset\mathbb{E}_{1}\subset\mathbb{X}_{0}\subset\mathbb{E}%
_{0}\qquad\text{and}\qquad\mathbb{X}_{1}\subset\mathbb{X}_{\alpha}%
\subset\mathbb{X}_{\beta}\subset\mathbb{X}_{0}, \label{BS}%
\end{equation}
whose precise relationships we describe below. The reasons for introducing
this curious arrangement of Banach spaces come from the beautiful
\emph{maximal regularity }construction of Da Prato and Grisvard \cite{DG79}.
As motivation, consider a linear initial value problem%
\[
(\ast)\left\{
\begin{array}
[c]{l}%
\dot{u}(t)=Lu(t)+f(t),\\
u(0)=u_{0},
\end{array}
\right.
\]
posed on a Banach space $\mathbb{Y}_{0}$, where the linear operator $L$
generates a strongly-continuous analytic semigroup on $\mathcal{L}%
(\mathbb{Y}_{0},\mathbb{Y}_{0})$. Then $L$ is a densely-defined closed
operator whose domain, when equipped with the graph norm,\footnote{The
\emph{graph norm with respect to }$\mathbb{Y}_{0}$ of a suitable linear
operator $L:D(L)\subseteq\mathbb{Y}_{0}\rightarrow\mathbb{Y}_{0}$ is
$\left\Vert x\right\Vert _{D(L)}=\left\Vert x\right\Vert _{\mathbb{Y}_{0}%
}+\left\Vert Lx\right\Vert _{\mathbb{Y}_{0}}$.} naturally becomes a Banach
space, which we may call $\mathbb{Y}_{1}$. Thus $\mathbb{Y}_{1}\hookrightarrow
\mathbb{Y}_{0}$ is a continuous and dense inclusion. Now (for some fixed
$T>0$) suppose that $f:[0,T]\rightarrow\mathbb{Y}_{0}$ is a bounded continuous
function. Then the \emph{formal }solution of $(\ast)$ is given by the integral
formula%
\[
u(t)=e^{tL}u_{0}+\int_{0}^{t}e^{(t-s)L}f(t)\,\mathrm{ds}\,,\qquad0<t\leq T.
\]
However, the convolution term above may not have enough regularity to make
this formal solution rigorous unless one imposes stronger hypotheses (like
requiring, for example, that $f:[0,T]\rightarrow\mathbb{Y}_{1}$ be bounded and
continuous). Da Prato and Grisvard's theory overcomes this difficulty by
restriction to suitably chosen interpolation spaces. In these (necessarily
non-reflexive) spaces, $\dot{u}$ and $Lu$ enjoy the same regularity as $f$,
which justifies the language \textquotedblleft maximal
regularity\textquotedblright.

Stability theory for nonlinear autonomous initial value problems%
\[
(\ast\ast)\left\{
\begin{array}
[c]{l}%
\dot{u}(t)=N(u(t)),\\
u(0)=u_{0},
\end{array}
\right.
\]
near a stationary solution $N(0)=0$ often proceeds by linearization, in which
one replaces $(\ast\ast)$ by $(\ast)$ with $L=N^{\prime}(0)$ and
$f(t)=N(u(t))-Lu(t)$. Here the regularity of $f$ is of evident importance.
This is where one exploits the (\ref{BS}) hierarchy. Suppose that (the
complexification of) $L:\mathbb{E}_{1}\rightarrow\mathbb{E}_{0}$ is a
sectorial operator.\footnote{See Definition~\ref{DefineSectorial} on
page~\pageref{DefineSectorial}.} The arguments Simonett uses to prove
Theorem~\ref{SimonettRaw} below require the linear problem $(\ast)$ to have a
bounded continuous solution $u:[0,T]\rightarrow\mathbb{X}_{1}$ for every
bounded continuous $f:[0,T]\rightarrow\mathbb{X}_{0}$. This will fail for
general sectorial $\mathbb{E}_{1}\rightarrow\mathbb{E}_{0}$ but can be
achieved by restricting $L:\mathbb{X}_{1}\rightarrow\mathbb{X}_{0}$ to
well-chosen continuous interpolation spaces $\mathbb{X}_{1}\hookrightarrow
\mathbb{X}_{0}$.

A similar stability theorem is derived by Da Prato and Lunardi \cite{DL85}. We
use Simonett's theorem because it takes optimal advantage of the parabolic
smoothing properties of the quasilinear equation (\ref{AQPE}) in continuous
interpolation spaces \cite{Simonett95}. With regard to (\ref{BS}), these
properties allow one to show that invariant manifolds can be exponentially
attractive in the norm of the smaller space $\mathbb{X}_{1}$ for solutions
whose initial data are close to a fixed point in the larger interpolation
space $\mathbb{X}_{\alpha}$. In particular, such solutions immediately
regularize and belong to $\mathbb{X}_{1}$ for all $t>0$.

With this brief introduction in hand, we are ready to state the seven
hypotheses one needs to apply the theorem. In our applications, we choose
certain little-H\"{o}lder spaces whose properties,
e.g.~(\ref{InterpolationIsomorphism}), greatly simplify verification of the
hypotheses. While perusing those hypotheses, readers may wish to consult
(\ref{DefineNestedSpaces}) and (\ref{YetMoreInterpolationSpaces}) below for
specifics of how the various spaces are realized in the remainder of this paper.

\begin{axiom}
\label{SimFirst}$\mathbb{X}_{1}\hookrightarrow\mathbb{X}_{0}$ and
$\mathbb{E}_{1}\hookrightarrow\mathbb{E}_{0}$ are continuous dense inclusions
of Banach spaces. For fixed $0<\beta<\alpha<1$, $\mathbb{X}_{\alpha}$ and
$\mathbb{X}_{\beta}$ are continuous interpolation spaces corresponding to the
inclusion $\mathbb{X}_{1}\hookrightarrow\mathbb{X}_{0}$.
\end{axiom}

\begin{axiom}
There is an autonomous quasilinear parabolic equation%
\begin{equation}
\frac{\partial}{\partial\tau}\widetilde{\mathbf{g}}(\tau)\mathbf{=Q}%
(\widetilde{\mathbf{g}}(\tau)),\qquad(\tau\geq0), \label{AQPE}%
\end{equation}
with the property that there exists a positive integer $k$ such that for all
$\widehat{\mathbf{g}}$ in some open set $\mathbb{G}_{\beta}\subseteq
\mathbb{X}_{\beta}$, the domain $\mathbb{D}(\mathbf{L}_{\widehat{\mathbf{g}}%
})$ of $\mathbf{L}_{\widehat{\mathbf{g}}}$ contains $\mathbb{X}_{1}$ and the
map $\widehat{\mathbf{g}}\mapsto\mathbf{L}_{\widehat{\mathbf{g}}}%
|_{\mathbb{X}_{1}}$ belongs to $C^{k}(\mathbb{G}_{\beta},\mathcal{L}%
(\mathbb{X}_{1},\mathbb{X}_{0}))$.
\end{axiom}

\begin{axiom}
For each $\widehat{\mathbf{g}}\in\mathbb{G}_{\beta}$, there exists an
extension $\widehat{\mathbf{L}}_{\widehat{\mathbf{g}}}$ of $\mathbf{L}%
_{\widehat{\mathbf{g}}}$ to a domain $\widehat{\mathbb{D}}(\widehat
{\mathbf{g}})$ that contains $\mathbb{E}_{1}$ (hence is dense in
$\mathbb{E}_{0}$).
\end{axiom}

\begin{axiom}
\label{Sectorial}For each $\widehat{\mathbf{g}}\in\mathbb{G}_{\alpha
}=\mathbb{G}_{\beta}\cap\mathbb{X}_{\alpha}$, $\widehat{\mathbf{L}}%
_{\widehat{\mathbf{g}}}|_{\mathbb{E}_{1}}\in\mathcal{L}(\mathbb{E}%
_{1},\mathbb{E}_{0})$ generates a strongly-continuous analytic semigroup on
$\mathcal{L}(\mathbb{E}_{0},\mathbb{E}_{0})$. (Observe that for $\widehat
{\mathbf{g}}\in\mathbb{G}_{\alpha}$, this implies that $\widehat{\mathbb{D}%
}(\widehat{\mathbf{g}})$ becomes a Banach space when equipped with the graph
norm with respect to $\mathbb{E}_{0}$.)
\end{axiom}

\begin{axiom}
For each $\widehat{\mathbf{g}}\in\mathbb{G}_{\alpha}$, $\mathbf{L}%
_{\widehat{\mathbf{g}}}$ is the part of $\widehat{\mathbf{L}}_{\widehat
{\mathbf{g}}}$ in $\mathbb{X}_{0}$.\footnote{If $\mathbb{X}$ is a Banach space
with subspace $\mathbb{Y}$ and $L:D(L)\subseteq\mathbb{X}\rightarrow
\mathbb{X}$ is linear, then $L^{\mathbb{Y}}$, \emph{the part of }$L$\emph{ in
}$\mathbb{Y}$, is defined by the action $L^{\mathbb{Y}}:x\mapsto Lx$ on the
domain $D(L^{\mathbb{Y}})=\{x\in D(L):Lx\in\mathbb{Y\}}$.}
\end{axiom}

\begin{axiom}
\label{CIspaces}For each $\widehat{\mathbf{g}}\in\mathbb{G}_{\alpha}$, there
exists $\theta\in(0,1)$ such that $\mathbb{X}_{0}\cong(\mathbb{E}_{0}%
,\widehat{\mathbb{D}}(\widehat{\mathbf{g}}))_{\theta}$ and $\mathbb{X}%
_{1}\cong(\mathbb{E}_{0},\widehat{\mathbb{D}}(\widehat{\mathbf{g}}%
))_{1+\theta}$, where $(\mathbb{E}_{0},\widehat{\mathbb{D}}(\widehat
{\mathbf{g}}))_{1+\theta}=\{\mathbf{g}\in\widehat{\mathbb{D}}(\widehat
{\mathbf{g}}):\widehat{\mathbf{L}}_{\widehat{\mathbf{g}}}(\mathbf{g}%
)\in(\mathbb{E}_{0},\widehat{\mathbb{D}}(\widehat{\mathbf{g}}))_{\theta}\}$ as
a set, endowed with the graph norm of $\widehat{\mathbf{L}}_{\widehat
{\mathbf{g}}}$ with respect to $(\mathbb{E}_{0},\widehat{\mathbb{D}}%
(\widehat{\mathbf{g}}))_{\theta}$.\footnote{To see that this assumption is
equivalent to condition (\textsc{ii}) in \cite[Section~4]{Simonett95}, recall
the fact (see e.g. \cite[Propositions~2.2.2 and 2.2.4]{Lunardi95}) that if $A$
is the generator of a $C_{0}$ analytic semigroup in a Banach space
$\mathbb{X}$ and $0<\theta<1$, then $(X,\mathbb{D}(A))_{\theta}\cong%
\mathbb{D}_{A}(\theta)=\{x\in\mathbb{X}:\lim_{t\searrow0}t^{-\theta}%
(e^{tA}x-x)=0\}$.}
\end{axiom}

\begin{axiom}
\label{SimLast}$\mathbb{E}_{1}\hookrightarrow\mathbb{X}_{\beta}\hookrightarrow
\mathbb{E}_{0}$ is a continuous and dense inclusion such that there exist
$C>0$ and $\delta\in\left(  0,1\right)  $ such that for all $\eta\in
\mathbb{E}_{1}$, one has
\[
\left\Vert \eta\right\Vert _{\mathbb{X}_{\beta}}\leq C\left\Vert
\eta\right\Vert _{\mathbb{E}_{0}}^{1-\delta}\left\Vert \eta\right\Vert
_{\mathbb{E}_{1}}^{\delta}.
\]

\end{axiom}

Simonett obtains (a more general version of) the following result:

\begin{theorem}
[Simonett]\label{SimonettRaw}Let $\mathbf{L}_{\mathbf{g}}^{\mathbb{C}}$ denote
the complexification of the linearization $\mathbf{L}_{\mathbf{g}}$ of
(\ref{AQPE}) at a stationary solution $\mathbf{g}$ of (\ref{AQPE}%
).\footnote{Note that $\mathbf{L}_{\mathbf{g}}$ is the operator that appears
in Assumption~2.} Suppose there exists $\lambda_{\mathrm{s}}>0$ such that the
spectrum $\mathbf{\sigma}$ of $\mathbf{L}_{\mathbf{g}}^{\mathbb{C}}$ admit the
decomposition $\sigma=\sigma_{\mathrm{s}}\cup\{0\}$, where $0\ $is an
eigenvalue of finite multiplicity and $\sigma_{\mathrm{s}}\subseteq\left\{
z:\operatorname{Re}z\leq-\lambda_{\mathrm{s}}\right\}  $. If
Assumptions~\ref{SimFirst}--\ref{SimLast} hold, then:

\begin{enumerate}
\item For each $\alpha\in\lbrack0,1]$, there is a direct-sum decomposition
$\mathbb{X}_{\alpha}=\mathbb{X}_{\alpha}^{\mathrm{s}}\oplus\mathbb{X}_{\alpha
}^{\mathrm{c}}$, where $\mathbb{X}_{\alpha}^{\mathrm{c}}$ is the
finite-dimensional algebraic eigenspace corresponding to the null eigenvalue
of $\mathbf{L}_{\mathbf{g}}^{\mathbb{C}}$.

\item For each $r\in\mathbb{N}$, there exists $d_{r}>0$ such that for all
$d\in(0,d_{r}]$, there exists a bounded $C^{r}$ map $\gamma_{d}^{r}%
:B(\mathbb{X}_{1}^{\mathrm{c}},\mathbf{g},d)\rightarrow\mathbb{X}%
_{1}^{\mathrm{s}}$ such that $\gamma_{d}^{r}(\mathbf{g})=0$ and $D\gamma
_{d}^{r}(\mathbf{g})=0$. The image of $\gamma_{d}^{r}$ lies in the closed ball
$\bar{B}(\mathbb{X}_{1}^{\mathrm{s}},\mathbf{g},d)$. Its graph is a local
$C^{r}$ center manifold $\Gamma_{\mathrm{loc}}^{r}=\{(h,\gamma_{d}%
^{r}(h)):h\in B(\mathbb{X}_{1}^{\mathrm{c}},\mathbf{g},d)\}\subset
\mathbb{X}_{1}$ satisfying $T_{\mathbf{g}}\Gamma_{\mathrm{loc}}^{r}%
\cong\mathbb{X}_{1}^{\mathrm{c}}$. Moreover, $\Gamma_{\mathrm{loc}}^{r}$ is
invariant for solutions of (\ref{AQPE}) as long as they remain in
$B(\mathbb{X}_{1}^{\mathrm{c}},\mathbf{g},d)\times B(\mathbb{X}_{1}%
^{\mathrm{s}},0,d)$.

\item Fix $\lambda\in(0,\lambda_{\mathrm{s}})$. Then for each $\alpha\in
(0,1)$, there exist $C>0$ and $d\in(0,d_{r}]$ such that for each initial datum
$\widetilde{\mathbf{g}}(0)\in B(\mathbb{X}_{\alpha},\mathbf{g},d)$ and all
times $\tau\geq0$ such that $\widetilde{\mathbf{g}}(\tau)\in B(\mathbb{X}%
_{\alpha},\mathbf{g},d)$, the center manifold $\Gamma_{\mathrm{loc}}^{r}$ is
exponentially attractive in the stronger space $\mathbb{X}_{1}$ in the sense
that
\[
\left\Vert \pi^{\mathrm{s}}\widetilde{\mathbf{g}}(\tau)-\gamma_{d}^{r}%
(\pi^{\mathrm{c}}\widetilde{\mathbf{g}}(\tau))\right\Vert _{\mathbb{X}_{1}%
}\leq\frac{C_{\alpha}}{\tau^{1-\alpha}}e^{-\lambda\tau}\left\Vert
\pi^{\mathrm{s}}\widetilde{\mathbf{g}}(0)-\gamma_{d}^{r}(\pi^{\mathrm{c}%
}\widetilde{\mathbf{g}}(0))\right\Vert _{\mathbb{X}_{\alpha}}.
\]
Here, $\widetilde{\mathbf{g}}(\tau)$ is the unique solution of (\ref{AQPE}),
while $\pi^{\mathrm{s}}$ and $\pi^{\mathrm{c}}$ denote the projections onto
$\mathbb{X}_{\alpha}^{\mathrm{s}}\cong(\mathbb{X}_{1}^{\mathrm{s}}%
,\mathbb{X}_{0}^{\mathrm{s}})_{\alpha}$ and $\mathbb{X}_{\alpha}^{\mathrm{c}}%
$, respectively.
\end{enumerate}
\end{theorem}

\subsection{Prerequisites for application of the theorem\label{Setup}}

We now establish the context in which we apply Theorem~\ref{SimonettRaw}.

Let $\widetilde{\mathbf{g}}_{\tau}=\mathbf{Q}(\widetilde{\mathbf{g}})$ denote
the rescaled $\mathbb{R}^{N}$-invariant Ricci flow system (\ref{RRFS}),
modified by DeTurck diffeomorphisms with respect to a background metric
$\mathbf{\underline{\mathbf{g}}}$ chosen as in Sections~\ref{KappaFlow}%
--\ref{MTRF} below. We assume that $\mathbf{g}$ is a smooth stationary
solution, namely that $\mathbf{g}_{\tau}=\mathbf{Q}(\mathbf{g})=0$.

For fixed $0<\sigma<\rho<1$, consider the following nested spaces:%
\begin{equation}%
\begin{array}
[c]{ccc}%
\mathbb{E}_{0} & = & \Sigma_{0}^{0+\sigma}(\mathcal{Y})\times\Sigma
_{1}^{0+\sigma}(\mathcal{Y})\times\Sigma_{2}^{0+\sigma}(\mathcal{Y})\\
\cup &  & \\
\mathbb{X}_{0} & = & \Sigma_{0}^{0+\rho}(\mathcal{Y})\times\Sigma_{1}^{0+\rho
}(\mathcal{Y})\times\Sigma_{2}^{0+\rho}(\mathcal{Y})\\
\cup &  & \\
\mathbb{\mathbb{E}}_{1} & = & \Sigma_{0}^{2+\sigma}(\mathcal{Y})\times
\Sigma_{1}^{2+\sigma}(\mathcal{Y})\times\Sigma_{2}^{2+\sigma}(\mathcal{Y})\\
\cup &  & \\
\mathbb{X}_{1} & = & \Sigma_{0}^{2+\rho}(\mathcal{Y})\times\Sigma_{1}^{2+\rho
}(\mathcal{Y})\times\Sigma_{2}^{2+\rho}(\mathcal{Y})
\end{array}
. \label{DefineNestedSpaces}%
\end{equation}

For fixed $1/2\leq\beta<\alpha<1$ and $\varepsilon>0$ to be chosen, define%
\begin{equation}
\mathbb{X}_{\beta}=(\mathbb{X}_{0},\mathbb{X}_{1})_{\beta}\qquad
\text{and}\qquad\mathbb{X}_{\alpha}=(\mathbb{X}_{0},\mathbb{X}_{1})_{\alpha}
\label{YetMoreInterpolationSpaces}%
\end{equation}
and%
\begin{equation}
\mathbb{G}_{\beta}=B(\mathbb{X}_{\beta},\mathbf{g},\varepsilon)\qquad
\text{and}\qquad\mathbb{G}_{\alpha}=\mathbb{G}_{\beta}\cap\mathbb{X}_{\alpha}.
\label{SmallBalls}%
\end{equation}

For each $\mathbf{\hat{g}}\in\mathbb{G}_{\beta}$, let $\widehat{\mathbf{L}%
}_{\widehat{\mathbf{g}}}$ denote the linearization of $\mathbf{Q}$ at
$\mathbf{\hat{g}}$, regarded as an unbounded linear operator on $\mathbb{E}%
_{0}$ with dense domain $\widehat{\mathbb{D}}(\widehat{\mathbf{g}}%
)=\mathbb{E}_{1}$. Let $\mathbf{L}_{\mathbf{\hat{g}}}$ denote the
corresponding operator on $\mathbb{X}_{0}$ with dense domain $\mathbb{D}%
(\mathbf{L}_{\widehat{\mathbf{g}}})=\mathbb{X}_{1}$.

Recall the following:

\begin{definition}
\label{DefineSectorial}A densely-defined linear operator $\mathbf{L}$ on a
complex Banach space $\mathbb{X}\neq\{0\}$ is said to be \emph{sectorial} if
there exist $\alpha\in(\pi,2\pi)$, $\omega\in\mathbb{R}$, and $C>0$ such that
the \textquotedblleft sector\textquotedblright%
\begin{equation}
S_{\alpha,\omega}=\{\lambda\in\mathbb{C}:\lambda\neq\omega,\ \left\vert
\arg(\lambda-\omega)\right\vert <\alpha\}
\end{equation}
is contained in the resolvent set $\rho(\mathbf{L})$, and such that%
\begin{equation}
\left\Vert (\lambda\mathbf{I}-\mathbf{L})^{-1}\right\Vert _{\mathcal{L}%
(\mathbb{X},\mathbb{X})}\leq\frac{C}{\left\vert \lambda-\omega\right\vert }
\label{Kato}%
\end{equation}
holds for all $\lambda\in S_{\alpha,\omega}\subseteq\rho(\mathbf{L})$.
\end{definition}

The following lemmas verify the technical hypotheses needed for
Theorem~\ref{SimonettRaw}.

\medskip

\begin{lemma}
~\label{Maps}

\begin{enumerate}
\item $\mathbf{\hat{g}}\mapsto\mathbf{L}_{\mathbf{\hat{g}}}$ is an analytic
map $\mathbb{G}_{\beta}\rightarrow\mathcal{L}(\mathbb{X}_{1},\mathbb{X}_{0})$.

\item $\mathbf{\hat{g}}\mapsto$ $\widehat{\mathbf{L}}_{\widehat{\mathbf{g}}}$
is an analytic map $\mathbb{G}_{\alpha}\rightarrow\mathcal{L}(\mathbb{E}%
_{1},\mathbb{E}_{0})$.

\item If $\widehat{\mathbf{L}}_{\mathbf{g}}^{\mathbb{C}}$ is sectorial, then
there exists $\varepsilon>0$ such that for all $\mathbf{\hat{g}}$ in the set
$\mathbb{G}_{\alpha}$ defined by (\ref{SmallBalls}), $\widehat{\mathbf{L}%
}_{\mathbf{\hat{g}}}^{\mathbb{C}}$ is the infinitesimal generator of an
analytic $C_{0}$-semigroup on $\mathcal{L}(\mathbb{E}_{0},\mathbb{E}_{0})$.
\end{enumerate}
\end{lemma}

\begin{proof}
Statements~(1)--(2) are proved in Lemma~3.3 of \cite{GIK02}.

To prove statement~(3), first observe that $\widehat{\mathbf{L}}_{\mathbf{g}%
}^{\mathbb{C}}$ generates an analytic $C_{0}$-semigroup on $\mathcal{L}%
(\mathbb{E}_{0},\mathbb{E}_{0})$: it is a standard fact that a sectorial
operator generates an analytic semigroup; one knows that semigroup is strongly
continuous because $\widehat{\mathbf{L}}_{\mathbf{g}}^{\mathbb{C}}$ is densely
defined by construction. Now if $\mathbf{\hat{g}}\in\mathbb{G}_{\alpha}$, then
by statement~(2), we can choose $\varepsilon>0$ small enough so that
\[
\left\Vert \widehat{\mathbf{L}}_{\mathbf{\hat{g}}}^{\mathbb{C}}-\widehat
{\mathbf{L}}_{\mathbf{g}}^{\mathbb{C}}\right\Vert _{\mathcal{L}(\mathbb{E}%
_{1},\mathbb{E}_{0})}<\frac{1}{C+1},
\]
where $C>0$ is the constant in (\ref{Kato}) corresponding to $\widehat
{\mathbf{L}}_{\mathbf{g}}^{\mathbb{C}}$. As is well known, this implies that
$\widehat{\mathbf{L}}_{\mathbf{\hat{g}}}^{\mathbb{C}}$ is sectorial as well.
(See \cite[Proposition~2.4.2]{Lunardi95}, for example.) Statement~(3) follows easily.
\end{proof}

\begin{lemma}
\label{JustDoIt}If $\widehat{\mathbf{L}}_{\mathbf{g}}^{\mathbb{C}}$ is
sectorial, then the choices made in (\ref{DefineNestedSpaces}%
)--(\ref{SmallBalls}) ensure that Assumptions~\ref{SimFirst}--\ref{SimLast}
are satisfied for the system $\widetilde{\mathbf{g}}_{\tau}=\mathbf{Q}%
(\widetilde{\mathbf{g}})$ given in (\ref{RRFS}).
\end{lemma}

\begin{proof}
Assumption~1 holds by construction. Assumption~2 follows directly from
statement~(2) of Lemma~\ref{Maps}, because the \textquotedblleft off-diagonal
terms\textquotedblright\ in (\ref{RRFS}) are all contractions involving at
most one derivative of $\widetilde{\mathbf{g}}$. Assumptions~3 and 5 hold by
construction of $\widehat{\mathbf{L}}_{\mathbf{\hat{g}}}$ and $\mathbf{L}%
_{\mathbf{\hat{g}}}$. Assumption~\ref{Sectorial} is a consequence of the
hypothesis on $\widehat{\mathbf{L}}_{\mathbf{g}}^{\mathbb{C}}$ and
statement~(3) of Lemma~\ref{Maps}. To verify Assumption~6, first observe that
putting $\theta=\left(  \rho-\sigma\right)  /2\in\left(  0,1\right)  $ into
the isomorphism identity (\ref{InterpolationIsomorphism}) yields
$\mathbb{X}_{0}\cong\left(  \mathbb{E}_{0},\mathbb{E}_{1}\right)  _{\theta}$;
then recall that standard Schauder theory implies that the graph norm of
$\widehat{\mathbf{L}}_{\widehat{\mathbf{g}}}$ with respect to $\mathbb{E}_{0}$
is equivalent to the $\mathbb{E}_{1}$ norm. (So the graph norm of
$\mathbf{L}_{\mathbf{\hat{g}}}$ with respect to $\mathbb{X}_{0}$ is equivalent
to the $\mathbb{X}_{1}$ norm.) Finally, Assumption~7 is implied by
(\ref{InterpolationNorm}).
\end{proof}

\section{The case $N=1$: $\varkappa$-rescaled locally $\mathbb{R}^{1}%
$-invariant Ricci flow\label{KappaFlow}}

To consider evolving Riemannian metrics on a bundle $\mathbb{R}^{1}%
\hookrightarrow\mathcal{M}\overset{\pi}{\mathcal{\longrightarrow}}%
\mathcal{B}^{n}$, it is convenient to change notation. Let $(x^{1}%
,\ldots,x^{n})$ be local coordinates on $\mathcal{B}^{n}$ and let $x^{0}$
denote the coordinate on an $\mathbb{R}^{1}$ fiber. If $(\mathcal{M}%
,\mathbf{\bar{g}}(t):t\in\mathcal{I})$ is a locally $\mathbb{R}^{1}$-invariant
solution (\ref{g-bar}), then there exist a Riemannian metric $\bar{g}$, a
$1$-form $\bar{A}$, and a function $\bar{u}$, all defined on $\mathcal{B}%
^{n}\times\mathcal{I}$, such that we can write $\mathbf{\bar{g}}$ in local
coordinates as
\begin{equation}
\mathbf{\bar{g}}=e^{2\bar{u}}\,dx^{0}\otimes dx^{0}+e^{2\bar{u}}\bar{A}%
_{k}\,(dx^{0}\otimes dx^{k}+dx^{k}\otimes dx^{0})+(e^{2\bar{u}}\bar{A}_{i}%
\bar{A}_{j}+\bar{g}_{ij})\,dx^{i}\otimes dx^{j}. \label{KaluzaKleinMetric}%
\end{equation}
We again abuse notation by writing $\mathbf{\bar{g}}(t)=(\bar{g}(t),\bar
{A}(t),\bar{u}(t))$. In the coordinates of this section, system~(\ref{RFS})
takes the form
\begin{subequations}
\label{N=1Unscaled}%
\begin{align}
\frac{\partial}{\partial t}\bar{g}_{ij}  &  =-2\bar{R}_{ij}+e^{2\bar{u}}%
\bar{g}^{k\ell}(d\bar{A})_{ik}(d\bar{A})_{j\ell}+2\bar{\nabla}_{i}u\bar
{\nabla}_{j}u,\\
\frac{\partial}{\partial t}\bar{A}_{k}  &  =\bar{g}^{ij}\{\bar{\nabla}%
_{i}(d\bar{A})_{jk}+2(\bar{\nabla}_{i}u)(d\bar{A})_{jk}\},\\
\frac{\partial}{\partial t}\bar{u}  &  =\bar{\Delta}u-\frac{1}{4}e^{2\bar{u}%
}|d\bar{A}|^{2}.
\end{align}
We are interested in the stability of metrics of this type on trivial
(i.e.~product) bundles. It turns out that the most practical rescaling depends
on the sign of the curvatures of $(\mathcal{B},g)$. We explain this in
Remarks~\ref{kappa-bad}--\ref{volume-bad} below.

\subsection{$\varkappa$-rescaled flow}

Suppose that $(\mathcal{M},\mathbf{\bar{g}})$ is a Riemannian product over an
Einstein manifold with $2\operatorname*{Rc}(\bar{g}(-\varkappa))=\varkappa
\bar{g}(-\varkappa)$ for some $\varkappa=\pm1$. Then the renormalization in
Example~\ref{Product} yields
\end{subequations}
\begin{subequations}
\label{kappa-rescaled}%
\begin{align}
\frac{\partial}{\partial\tau}g_{ij}  &  =-2R_{ij}+e^{2u}g^{k\ell}%
(dA)_{ik}(dA)_{j\ell}+2\nabla_{i}u\nabla_{j}u+\varkappa g_{ij},\\
\frac{\partial}{\partial\tau}A_{k}  &  =g^{ij}\{\nabla_{i}(dA)_{jk}%
+2(\nabla_{i}u)(dA)_{jk}\}+\frac{\varkappa}{2}A_{k},\\
\frac{\partial}{\partial\tau}u  &  =\Delta u-\frac{1}{4}e^{2u}|dA|^{2}.
\end{align}
We call this system (in which all quantities are computed with respect to $g$)
the $\varkappa$\emph{-rescaled locally $\mathbb{R}^{1}$-invariant Ricci flow.}
It is most satisfactory for studying the case $\varkappa=-1$. (In
Section~\ref{VolumeFlow}, we apply a different rescaling that is more suitable
for the case $\varkappa=+1$.)

\subsection{Linearization at a stationary solution of $\varkappa$-rescaled
flow}

Any Riemannian product $(\mathbb{R}\times\mathcal{B},e^{2u}\,dx^{0}\otimes
dx^{0}+g_{ij}\,dx^{i}\otimes dx^{j})$ with $2\operatorname*{Rc}(g)=\varkappa
g$, $A$ identically zero, and $u$ constant in space is clearly a stationary
solution of (\ref{kappa-rescaled}). Let $(g+h,B,u+v)$ denote a perturbation of
such a fixed point $\mathbf{g}=(g,0,u)$. Define $H=\operatorname*{tr}\!_{g}h$,
and let $\Delta_{\ell}$ denote the Lichnerowicz Laplacian acting on symmetric
$(2,0)$-tensor fields. In coordinates,%
\end{subequations}
\begin{equation}
\Delta_{\ell}h_{ij}=\Delta h_{ij}+2R_{ipqj}h^{pq}-R_{i}^{k}h_{kj}-R_{j}%
^{k}h_{ik}. \label{Lichnerowicz}%
\end{equation}

\begin{lemma}
\label{FirstLinearization}The linearization of (\ref{kappa-rescaled}) at a
fixed point $\mathbf{g}=(g,0,u)$ with $2\operatorname*{Rc}=\varkappa g$ and
$u$ constant acts on $(h,B,v)$ by
\begin{subequations}
\label{kappa-linear-nonparabolic}%
\begin{align}
\frac{\partial}{\partial\tau}h_{ij}  &  =\Delta_{\ell}h_{ij}+\{\nabla
_{i}(\delta h)_{j}+\nabla_{j}(\delta h)_{i}+\nabla_{i}\nabla_{j}H\}+\varkappa
h_{ij},\\
\frac{\partial}{\partial\tau}B_{k}  &  =-(\delta dB)_{k}+\frac{\varkappa}%
{2}B_{k},\\
\frac{\partial}{\partial\tau}v  &  =\Delta v.
\end{align}

\end{subequations}
\end{lemma}

\begin{proof}
For the convenience of the reader, we begin by recalling a few classical
variation formulas (all of which may be found in \cite[Chapter~1,
Section~K]{Besse87}, for instance). Let $\tilde{g}(\varepsilon)$ be a smooth
one-parameter family of Riemannian metrics such that%
\[
\left.  \tilde{g}(0)=g\qquad\text{and}\qquad\frac{\partial}{\partial
\varepsilon}\right\vert _{\varepsilon=0}\tilde{g}=h.
\]
Using tildes to denote geometric quantities associated to $\tilde{g}$ and
undecorated characters to denote quantities associated to $g$, one computes
that%
\begin{align*}
\left.  \frac{\partial}{\partial\varepsilon}\right\vert _{\varepsilon=0}%
\tilde{g}^{ij}  &  =-h^{ij}\quad\left(  =-g^{ik}g^{j\ell}h_{k\ell}\right)  ,\\
\left.  \frac{\partial}{\partial\varepsilon}\right\vert _{\varepsilon=0}%
\tilde{\Gamma}_{ij}^{k}  &  =\frac{1}{2}(\nabla_{i}h_{j}^{k}+\nabla_{j}%
h_{i}^{k}-\nabla^{k}h_{ij}),\\
\left.  \frac{\partial}{\partial\varepsilon}\right\vert _{\varepsilon=0}%
\tilde{R}_{ij}  &  =-\frac{1}{2}\left\{  \Delta_{\ell}h_{ij}+\nabla_{i}(\delta
h)_{j}+\nabla_{j}(\delta h)_{i}+\nabla_{i}\nabla_{j}H\right\}  .
\end{align*}
The lemma follows by careful but straightforward application of these formulas.
\end{proof}

The linear system (\ref{kappa-linear-nonparabolic}) is autonomous but not
quite parabolic. We impose parabolicity by the DeTurck trick \cite{DT83,
DT03}. Fix a background connection $\underline{\Gamma}$ and define a
$1$-parameter family of vector fields $W(\tau)$ along a solution
$\mathbf{g}(\tau)$ of (\ref{kappa-rescaled}) by
\begin{subequations}
\label{DeTurck-Hyperbolic}%
\begin{align}
W^{0}  &  =\delta A,\\
W^{k}  &  =g^{ij}(\Gamma_{ij}^{k}-\underline{\Gamma}_{ij}^{k})\qquad(1\leq
k\leq n).
\end{align}
The solution of the $\varkappa$\emph{-rescaled }$\mathbb{R}^{1}$%
\emph{-invariant Ricci--DeTurck flow} is the $1$-parameter family of metrics
$\psi_{\tau}^{\ast}\mathbf{g}(\tau)$, where the diffeomorphisms $\psi_{\tau}$
are generated by $W(\tau)$, subject to the initial condition $\psi
_{0}=\operatorname*{id}$. In what follows, we take $\underline{\Gamma}$ to be
the Levi-Civita connection of the stationary solution around which we
linearize. Observe that a stationary solution $\mathbf{g}=(g,0,u)$ of
(\ref{kappa-rescaled}) with $2\operatorname*{Rc}=\varkappa g$ and $u$ constant
is then also a stationary solution of the $\varkappa$-rescaled Ricci--DeTurck flow.
\end{subequations}
\begin{lemma}
The linearization of the $\varkappa$-rescaled Ricci--DeTurck flow at a fixed
point $\mathbf{g}=(g,0,u)$ with $2\operatorname*{Rc}=\varkappa g$ is the
autonomous, self-adjoint, strictly parabolic system%
\[
\frac{\partial}{\partial\tau}%
\begin{pmatrix}
h & B & v
\end{pmatrix}
=\mathbf{L}%
\begin{pmatrix}
h & B & v
\end{pmatrix}
=%
\begin{pmatrix}
\mathbf{L}_{2}h & \mathbf{L}_{1}B & \mathbf{L}_{0}v
\end{pmatrix}
,
\]
where
\begin{subequations}
\label{kappa-linear-parabolic}%
\begin{align}
\mathbf{L}_{2}h  &  =\Delta_{\ell}h+\varkappa h,\label{Linear-h}\\
\mathbf{L}_{1}B  &  =\Delta_{1}B+\frac{\varkappa}{2}B,\label{Linear-b}\\
\mathbf{L}_{0}v  &  =\Delta_{0}v. \label{Linear-v}%
\end{align}
Here $-\Delta_{p}=d\delta+\delta d$ denotes the Hodge--de Rham Laplacian
acting on $p$-forms.
\end{subequations}
\end{lemma}

\begin{proof}
The Christoffel symbols $\mathbf{\Gamma}\ $for a fixed point $\mathbf{g}%
=(g,0,u)$ vanish if any index is zero and satisfy $\mathbf{\Gamma}_{ij}%
^{k}=\Gamma_{ij}^{k}$ otherwise, where $\Gamma_{ij}^{k}$ are the Christoffel
symbols for $g$. Hence the DeTurck correction terms to the linearization
(\ref{kappa-linear-nonparabolic}) of (\ref{kappa-rescaled}) are%
\begin{align*}
(\mathcal{L}_{W}\mathbf{g})_{00}  &  =0,\\
(\mathcal{L}_{W}\mathbf{g})_{0k}  &  =e^{2u}(d\delta B)_{k},\\
(\mathcal{L}_{W}\mathbf{g})_{ij}  &  =\nabla_{i}(\delta h)_{j}+\nabla
_{j}(\delta h)_{i}+\nabla_{i}\nabla_{j}H,
\end{align*}
whence the result is immediate from (\ref{KaluzaKleinMetric}) and
(\ref{kappa-linear-nonparabolic}).
\end{proof}

\subsection{Linear stability of $\varkappa$-rescaled flow}

Now let us make the stronger assumption that $\mathbf{g}=(g,0,u)$ is a fixed
point of the $\varkappa$-rescaled $\mathbb{R}^{1}$-invariant Ricci--DeTurck
flow with $u$ constant and $g$ a metric of constant sectional curvature
$-1/2(n-1)$.

We say a linear operator $L$ is \emph{weakly (strictly) stable }if its
spectrum is confined to the half plane $\operatorname{Re}z\leq0$ (and is
uniformly bounded away from the imaginary axis).

Because $\mathbf{L}$ is diagonal, we can determine its stability by examining
its component operators. The conclusions we obtain here will hold below when
we extend $\mathbf{L}$ to a complex-valued operator on a larger domain in
which smooth representatives are dense.

\begin{lemma}
\label{hyperbolic}Let $\mathbf{g}=(g,0,u)$ be a metric of the form
(\ref{KaluzaKleinMetric}) such that $u$ is constant and $g$ has constant
sectional curvature $-1/2(n-1)$. Then the linear system
(\ref{kappa-linear-parabolic}) has the following stability properties:

The operator $\mathbf{L}_{0}\mathbf{\ }$is weakly stable; constant functions
form its null eigenspace.

The operator $\mathbf{L}_{1}$ is strictly stable.

If $n\geq3$, then $\mathbf{L}_{2}$ is strictly stable.

If $n=2$, then $\mathbf{L}_{2}$ is weakly stable. On an orientable surface
$\mathcal{B}$ of genus $\gamma\geq2$, the null eigenspace of $\Delta_{\ell}-1$
is the $(6\gamma-6)$-dimensional space of holomorphic quadratic differentials.
\end{lemma}

\begin{proof}
The statements about $\mathbf{L}_{0}$ and $\mathbf{L}_{1}$ are clear.

Let $Q=Q(h)$ denote the contraction $Q=R_{ijk\ell}h^{i\ell}h^{jk}$ and define
a tensor field $T=T(h)$ by
\begin{equation}
T_{ijk}=\nabla_{k}h_{ij}-\nabla_{i}h_{jk}. \label{Define-T}%
\end{equation}
Then integrating by parts and applying Koiso's Bochner formula \cite{Koiso79},
one obtains%
\begin{align}
(\mathbf{L}_{2}h,h)  &  =-\left\Vert \nabla h\right\Vert ^{2}+2\int
_{\mathcal{B}}Q(h)\,\mathrm{d\mu}\,\nonumber\\
&  =-\frac{1}{2}\left\Vert T\right\Vert ^{2}-\left\Vert \delta h\right\Vert
^{2}-\frac{1}{2}\left\Vert h\right\Vert ^{2}+\int_{\mathcal{B}}%
Q(h)\,\mathrm{d\mu}\,\nonumber\\
&  =-\frac{1}{2}\left\Vert T\right\Vert ^{2}-\left\Vert \delta h\right\Vert
^{2}-\frac{1}{2(n-1)}\left\Vert H\right\Vert ^{2}-\frac{n-2}{2(n-1)}\left\Vert
h\right\Vert ^{2}. \label{HyperbolicComputation}%
\end{align}
This proves that (\ref{Linear-h}) is strictly stable for all $n>2$.

When $n=2$, it follows from (\ref{HyperbolicComputation}) that a tensor field
$h$ belongs to the nullspace of $\mathbf{L}_{2}=\Delta_{\ell}+R$ if and only
if $h$ is trace free and divergence free, with $T(h)=0$ vanishing identically.
We first verify that the latter condition is superfluous. In normal
coordinates at an arbitrary point $p\in\mathcal{B}$, all components of $T(h)$
vanish except possibly $T_{112}=-T_{211}=\partial_{2}h_{11}-\partial_{1}%
h_{12}$ and $T_{221}=-T_{122}=\partial_{1}h_{22}-\partial_{2}h_{21}$. But one
has $\partial_{1}h_{11}+\partial_{2}h_{21}=0$ and $\partial_{1}h_{12}%
+\partial_{2}h_{22}=0$ because $h$ is divergence free and $\partial_{1}%
h_{11}+\partial_{1}h_{22}=0$ and $\partial_{2}h_{11}+\partial_{2}h_{22}=0$
because $h$ is trace free. It follows easily that $T$ vanishes at $p$, hence everywhere.

Now let $h$ be a trace-free and divergence-free tensor field on an orientable
surface $\mathcal{B}$ of genus $\gamma\geq2$. To identify the nullspace of
$\mathbf{L}_{2}$, it is most convenient to work in local complex coordinates.
The trace-free condition implies that%
\[
h=f(z)\,dz\,dz+\bar{f}(z)\,d\bar{z}\,d\bar{z}%
\]
for some function $f:\mathcal{B}\rightarrow\mathbb{C}$, while the
divergence-free condition implies that $f$ is holomorphic. Hence $h$ is a
holomorphic quadratic differential. If $\mathcal{R}$ is a Riemann surface
obtained from a closed surface by deleting $p$ disjoint closed discs and $q$
isolated points not on any disc, the Riemann--Roch theorem implies that the
space of holomorphic quadratic differentials on $\mathcal{R}$ has real
dimension $6(\gamma-1)+3p+2q$. The result follows.
\end{proof}

When $n=2$, the nullspace for the system is classically identified with the
cotangent space to Teichm\"{u}ller space.

\subsection{Convergence and stability of $\varkappa$-rescaled flow}

We now prove Theorem~\ref{HyperbolicIsStable}. To obtain asymptotic stability
of $\varkappa$-rescaled $\mathbb{R}^{1}$-invariant Ricci flow
(\ref{N=1Unscaled}) with $\varkappa=-1$ from its Ricci--DeTurck linearization
(\ref{kappa-linear-parabolic}), we shall freely use the theory reviewed in
Section~\ref{Theory} and in particular the maximal regularity and
interpolation spaces chosen in Section~\ref{Setup}.

\begin{proof}
[Proof of Theorem~\ref{HyperbolicIsStable}]Our notation is as follows.
$\mathbf{g}=(g,A,u)$ is a locally $\mathbb{R}^{1}$-invariant metric
(\ref{KaluzaKleinMetric}) on a product $\mathbb{R}^{1}\times\mathcal{B}$, with
$\mathcal{B}$ compact and orientable, such that $g$ has constant sectional
curvature $-1/2(n-1)$, $A$ vanishes, and $u$ is constant. $\widehat
{\mathbf{g}}(\tau)$ is the unique solution of $\varkappa$-rescaled locally
$\mathbb{R}^{1}$-invariant Ricci--DeTurck flow corresponding to an initial
datum $\widetilde{\mathbf{g}}(0)$. Then there are diffeomorphisms
$\varphi_{\tau}$ with $\varphi_{0}=\operatorname*{id}$ such that the unique
solution $\widetilde{\mathbf{g}}(\tau)$ of $\varkappa$-rescaled locally
$\mathbb{R}^{1}$-invariant Ricci flow (\ref{kappa-rescaled}) with initial
datum $\widetilde{\mathbf{g}}(0)$ is given by%
\[
\widetilde{\mathbf{g}}(\tau)=\varphi_{\tau}^{\ast}(\widehat{\mathbf{g}}%
(\tau)).
\]
The proof consists of four steps.

\textsc{Step 1.} We prove that (the complexification of) the linearization
$\mathbf{L}\equiv\mathbf{L}_{\mathbf{g}}$ of $\varkappa$-rescaled locally
$\mathbb{R}^{1}$-invariant Ricci--DeTurck flow is sectorial. Observe that
$\mathbf{L}$ is strictly elliptic and self adjoint. By Lemma~\ref{hyperbolic},
one may fix $\omega>0$ such that its spectrum $\sigma(\mathbf{L}^{\mathbb{C}%
})$ satisfies $\sigma(\mathbf{L}^{\mathbb{C}})\backslash\{0\}\subset
(-\infty,-\omega)$. Standard Schauder theory then implies that $\mathbf{L}%
^{\mathbb{C}}$ is sectorial. (See Lemma~3.4 in \cite{GIK02} for a more
detailed argument.)

\textsc{Step 2.} By Step~1 and Lemma~\ref{JustDoIt}, we may apply
Theorem~\ref{SimonettRaw} to $\widehat{\mathbf{g}}(\tau)$. Recall that
$\mathbb{X}_{\alpha}=(\mathbb{X}_{0},\mathbb{X}_{1})_{\alpha}$, where
$\mathbb{X}_{0}$ and $\mathbb{X}_{1}$ are defined by (\ref{DefineNestedSpaces}%
). Theorem~\ref{SimonettRaw} implies the following three statements:

\begin{enumerate}
\item For each $\alpha\in\lbrack0,1]$, there is a direct sum decomposition
$\mathbb{X}_{\alpha}=\mathbb{X}_{\alpha}^{\mathrm{s}}\oplus\mathbb{X}_{\alpha
}^{\mathrm{c}}$, where $\mathbb{X}_{\alpha}^{\mathrm{c}}$ is the nullspace of
$\mathbf{L}$ of dimension%
\[
\dim\mathbb{X}_{\alpha}^{\mathrm{c}}=\left\{
\begin{array}
[c]{cl}%
1+6(\gamma-1) & \text{if }n=2,\\
1 & \text{if }n\geq3.
\end{array}
\right.
\]
(Recall that if $n=2$, then $\mathcal{B}$ must be an orientable surface of
genus $\gamma\geq2$.)

\item For each $r\in\mathbb{N}$, there exist $d$ and a local $C^{r}$ center
manifold $\Gamma_{\mathrm{loc}}^{r}=\mathrm{graph}\left(  \gamma_{d}%
^{r}:B(\mathbb{X}_{1}^{\mathrm{c}},\mathbf{g},d)\rightarrow\mathbb{X}%
_{1}^{\mathrm{s}}\right)  $ which is invariant for solutions of $\ $%
Ricci--DeTurck flow as long as they remain in $B(\mathbb{X}_{1}^{\mathrm{c}%
},\mathbf{g},d)\times B(\mathbb{X}_{1}^{\mathrm{s}},0,d)$.

\item For each $\alpha\in(0,1)$, there exist positive constants $M=M(\alpha)$
and $d=d(\alpha,r)$ such that for each initial datum $\widetilde{\mathbf{g}%
}(0)\in B(\mathbb{X}_{\alpha},\mathbf{g},d)$ and all times $\tau>0$ such that
$\widehat{\mathbf{g}}(\tau)\in B(\mathbb{X}_{\alpha},\mathbf{g},d)$, the
center manifold $\Gamma_{\mathrm{loc}}^{r}$ is exponentially attractive in the
stronger space $\mathbb{X}_{1}$ in the sense that
\begin{equation}
\left\Vert \pi^{\mathrm{s}}\widehat{\mathbf{g}}(\tau)-\gamma_{d}^{r}%
(\pi^{\mathrm{c}}\widehat{\mathbf{g}}(\tau))\right\Vert _{\mathbb{X}_{1}}%
\leq\frac{M}{\tau^{1-\alpha}}e^{-\omega\tau}\left\Vert \pi^{\mathrm{s}%
}\widehat{\mathbf{g}}(0)-\gamma_{d}^{r}(\pi^{\mathrm{c}}\widehat{\mathbf{g}%
}(0))\right\Vert _{\mathbb{X}_{\alpha}}, \label{HyperbolicAttraction}%
\end{equation}
where $\pi^{\mathrm{s}}$ and $\pi^{\mathrm{c}}$ are projections onto
$\mathbb{X}_{\alpha}^{\mathrm{s}}\cong(\mathbb{X}_{1}^{\mathrm{s}}%
,\mathbb{X}_{0}^{\mathrm{s}})_{\alpha}$ and $\mathbb{X}_{\alpha}^{\mathrm{c}}%
$, respectively.
\end{enumerate}

\noindent(All constants introduced here and below implicitly depend on
$\mathbf{g}$.)

\textsc{Step 3.} We prove that the local center manifolds $\Gamma
_{\mathrm{loc}}^{r}$ coincide for all $r$, namely, that there is a unique
smooth center manifold $\Gamma=\mathrm{graph}\left(  \gamma:B(\mathbb{X}%
_{1}^{\mathrm{c}},\mathbf{g},d_{0})\rightarrow\mathbb{X}_{1}^{\mathrm{s}%
}\right)  $ consisting of fixed points of $\varkappa$-rescaled locally
$\mathbb{R}^{1}$-invariant Ricci flow (\ref{kappa-rescaled}). First observe
that for all $n\geq2$ and all $c\in\mathbb{R}$, $(g,0,c)$ is a stationary
solution both of (\ref{kappa-rescaled}) and $\ $Ricci--DeTurck flow. By
(\ref{HyperbolicAttraction}), any such metric sufficiently near $\mathbf{g}%
=(g,0,u)$ must belong to all $\Gamma_{\mathrm{loc}}^{r}$. When $n=2$,
Teichm\"{u}ller theory shows that there is a $6(\gamma-1)$-dimensional space
$\Gamma^{\prime}$ of metrics $\mathbf{g}^{\prime}=(g^{\prime},0,u)$ near
$\mathbf{g}$ such that $g^{\prime}$ is hyperbolic. Each such metric is a fixed
point of (\ref{kappa-rescaled}), hence evolves under Ricci--DeTurck flow only
by diffeomorphisms. By diffeomorphism invariance, $g^{\prime}(\tau)$ remains
hyperbolic, so $\mathbf{g}^{\prime}(\tau)$ remains in $\Gamma^{\prime}$. By
(\ref{HyperbolicAttraction}), $\mathbf{g}^{\prime}(\tau)$ must belong to all
$\Gamma_{\mathrm{loc}}^{r}$.

\textsc{Step 4.} Fix $\alpha\in(1/2,1-\rho/2)$. Then by the interpolation
isomorphism (\ref{InterpolationIsomorphism}), $\left\Vert \cdot\right\Vert
_{\mathbb{X}_{\alpha}}$ is equivalent to a $(1+\theta)$-H\"{o}lder norm, with
$\theta\in(\rho,1)$. For all $\widetilde{\mathbf{g}}(0)$ sufficiently near
$\mathbf{g}$, we now prove that $\widetilde{\mathbf{g}}(\tau)$ converges in
the $\mathbb{X}_{\alpha}$ norm to an element of $\Gamma$. (We give the proof
in general, even though certain steps are much easier when $n>2$.)

Fix $\lambda\in(0,\omega)$. Then by (\ref{HyperbolicAttraction}) and Step~3
there exists $C=C(M,\lambda)$ such that%
\begin{equation}
\left\Vert \pi^{\mathrm{s}}\widehat{\mathbf{g}}(\tau)-\gamma(\pi^{\mathrm{c}%
}\widehat{\mathbf{g}}(\tau))\right\Vert _{\mathbb{X}_{1}}\leq Ce^{-\lambda
\tau}\left\Vert \widetilde{\mathbf{g}}(0)-\mathbf{g}\right\Vert _{\mathbb{X}%
_{\alpha}} \label{SimpleHatHyperbolic}%
\end{equation}
for all $\tau>0$ such that $\widehat{\mathbf{g}}(\tau)\in B(\mathbb{X}%
_{\alpha},\mathbf{g},d)$. Let $\delta,\varepsilon$ be positive constants to be
determined so that $0<\varepsilon<\delta<d$, and suppose that $\widetilde
{\mathbf{g}}(0)\in B(\mathbb{X}_{\alpha},\mathbf{g},\varepsilon)$. Then it
follows from (\ref{SimpleHatHyperbolic}) and (\ref{kappa-rescaled}) that
\begin{equation}
\left\Vert \frac{\partial}{\partial\tau}\widetilde{\mathbf{g}}(\tau
)\right\Vert _{0+\rho}\leq C_{n}Ce^{-\lambda\tau}\varepsilon.
\label{SimpleTildeHyperbolic}%
\end{equation}
If $\delta/\varepsilon$ is sufficiently large, this implies that%
\[
\left\Vert \widetilde{\mathbf{g}}(\tau)-\mathbf{g}\right\Vert _{\mathbb{X}%
_{1}}\leq\varepsilon\left(  1+\frac{C_{n}C}{\lambda}\right)  <\delta,
\]
uniformly in time. Because $\widetilde{\mathbf{g}}(\tau)=\varphi_{\tau}^{\ast
}(\widehat{\mathbf{g}}(\tau))$, the only way that the solution $\widehat
{\mathbf{g}}(\tau)$ of Ricci--DeTurck flow could leave $B(\mathbb{X}_{\alpha
},\mathbf{g},\delta)$ is by diffeomorphisms. But as was observed by Hamilton
\cite{Ham95}, the diffeomorphisms $\varphi_{\tau}$ satisfy a harmonic map flow%
\[
\frac{\partial}{\partial\tau}\varphi_{\tau}=\Delta_{\widetilde{\mathbf{g}%
}(\tau),\mathbf{g}}\varphi_{\tau}.
\]
with domain metric $\widetilde{\mathbf{g}}(\tau)$ and codomain metric
$\mathbf{g}$. Because $\widetilde{\mathbf{g}}(\tau)\in B(\mathbb{X}%
_{1},\mathbf{g},\delta)$ for all $\tau\geq0$, mild generalizations of standard
estimates for harmonic map heat flow into negatively curved targets imply that
the diffeomorphisms $\varphi_{\tau}$ exist for all $\tau\geq0$ and satisfy
\[
\left\Vert \varphi_{\tau}-\operatorname*{id}\right\Vert _{\mathbb{X}_{1}}\leq
c
\]
for $c=c(\delta)$. (See \cite{ES64} and \cite{Simon83}, for instance.) Note
that $c$ does not increase if one makes $\delta$ smaller. It follows that for
$\delta\in(0,d)$ and $\varepsilon=\varepsilon(\delta)\in(0,\delta)$
sufficiently small, one has $\widehat{\mathbf{g}}(\tau)\in B(\mathbb{X}%
_{\alpha},\mathbf{g},d)$ for all $\tau\geq0$. By (\ref{SimpleTildeHyperbolic}%
), the result follows.
\end{proof}

\section{The case $N=1$: volume-rescaled locally $\mathbb{R}^{1}$-invariant
Ricci flow\label{VolumeFlow}}

In case $(\mathcal{M},\mathbf{\bar{g}})$ is a Riemannian product over an
Einstein manifold of positive Ricci curvature such that
\begin{equation}
\mathbf{\bar{g}}=e^{2\bar{u}}\,dx^{0}\otimes dx^{0}+e^{2\bar{u}}\bar{A}%
_{k}\,(dx^{0}\otimes dx^{k}+dx^{k}\otimes dx^{0})+(e^{2\bar{u}}\bar{A}_{i}%
\bar{A}_{j}+\bar{g}_{ij})\,dx^{i}\otimes dx^{j}, \label{KaluzaKleinMetric2}%
\end{equation}
the rescaled $\mathbb{R}^{1}$-invariant Ricci flow system (\ref{RRFS})
corresponding to the normalization in Example~\ref{NRF} is a more suitable choice.

\subsection{Volume-rescaled flow}

\emph{Volume-rescaled locally }$\mathbb{R}^{1}$\emph{-invariant Ricci flow} is
the system
\begin{subequations}
\label{volume-rescaled}%
\begin{align}
\frac{\partial}{\partial\tau}g_{ij}  &  =-2R_{ij}+e^{2u}g^{k\ell}%
(dA)_{ik}(dA)_{j\ell}+2\nabla_{i}u\nabla_{j}u+\frac{2}{n}(\oint_{\mathcal{B}%
}r\,\mathrm{d\mu}\,)g_{ij},\\
\frac{\partial}{\partial\tau}A_{k}  &  =g^{ij}\{\nabla_{i}(dA)_{jk}%
+2(\nabla_{i}u)(dA)_{jk}\}+\frac{1}{n}(\oint_{\mathcal{B}}r\,\mathrm{d\mu
}\,)A_{k},\\
\frac{\partial}{\partial\tau}u  &  =\Delta u-\frac{1}{4}e^{2u}|dA|^{2},
\end{align}
where all geometric quantities are computed with respect to $g$, and
\end{subequations}
\begin{equation}
r=R-\frac{1}{2}e^{2u}\left\vert dA\right\vert ^{2}-\left\vert \nabla
u\right\vert ^{2}.
\end{equation}

\subsection{Linearization at a stationary solution of volume-rescaled flow}

Fix a background connection $\underline{\Gamma}$ and define a $1$-parameter
family of vector fields $W(\tau)$ along a solution $\mathbf{g}(\tau)$ of
(\ref{volume-rescaled}) by%
\begin{align}
W^{0}  &  =\delta A,\\
W^{k}  &  =g^{ij}(\Gamma_{ij}^{k}-\underline{\Gamma})\qquad(1\leq k\leq
n).\nonumber
\end{align}
The solution of the \emph{volume-rescaled }$\mathbb{R}^{1}$\emph{-invariant
Ricci--DeTurck flow} is the $1$-parameter family of metrics $\psi_{\tau}%
^{\ast}\mathbf{g}(\tau)$, where the diffeomorphisms $\psi_{\tau}$ are
generated by $W(\tau)$, subject to the initial condition $\psi_{0}%
=\operatorname*{id}$. Below, we take $\underline{\Gamma}$ to be the
Levi-Civita connection of the metric about which we linearize.

Any Riemannian product $(\mathbb{R}\times\mathcal{B},e^{2u}\,dx^{0}\otimes
dx^{0}+g_{ij}\,dx^{i}\otimes dx^{j})$ with $g$ an Einstein metric, $A$
identically zero, and $u$ constant in space is clearly a stationary solution
of the volume-rescaled $\mathbb{R}^{1}$-invariant Ricci--DeTurck flow. Let
$\mathbf{g}=(g,0,u)$ be such a fixed point, with $\operatorname*{Rc}=Kg$ and
hence $r=nK$. Again, let $(g+h,B,u+v)$ denote a perturbation of $\mathbf{g}$.
We continue to write $H=\operatorname*{tr}\!_{g}h$ and to denote the
Lichnerowicz Laplacian (\ref{Lichnerowicz}) by $\Delta_{\ell}$. Proceeding as
we did for the $\varkappa$-rescaled flow, one obtains the following.

\begin{lemma}
The linearization of the volume-rescaled $\mathbb{R}^{1}$-invariant
Ricci--DeTurck flow at a fixed point $\mathbf{g}=(g,0,u)$ with
$\operatorname*{Rc}=Kg$ and $u$ constant is the autonomous, self-adjoint,
strictly parabolic system%
\[
\frac{\partial}{\partial\tau}%
\begin{pmatrix}
h & B & v
\end{pmatrix}
=\mathbf{L}%
\begin{pmatrix}
h & B & v
\end{pmatrix}
=%
\begin{pmatrix}
\mathbf{L}_{2}h & \mathbf{L}_{1}B & \mathbf{L}_{0}v
\end{pmatrix}
,
\]
where
\begin{subequations}
\label{volume-linear-parabolic}%
\begin{align}
\mathbf{L}_{2}h  &  =\Delta_{\ell}h+2K\{h-\frac{1}{n}\bar{H}%
g\},\label{Volume-h}\\
\mathbf{L}_{1}B  &  =\Delta_{1}B+KB,\label{Volume-B}\\
\mathbf{L}_{0}v  &  =\Delta_{0}v. \label{Volume-v}%
\end{align}
Here $\bar{H}=\oint_{\mathcal{B}}H\,\mathrm{d\mu}$ and $-\Delta_{p}%
=d\delta+\delta d$ denotes the Hodge--de Rham Laplacian acting on $p$-forms.
\end{subequations}
\end{lemma}

\begin{proof}
The argument here needs a few more variation formulas in addition to those we
recalled in the proof of Lemma~\ref{FirstLinearization}. Once again, let
$\tilde{g}(\varepsilon)$ be a smooth one-parameter family of Riemannian
metrics such that%
\[
\tilde{g}(0)=g\qquad\text{and}\qquad\left.  \frac{\partial}{\partial
\varepsilon}\right\vert _{\varepsilon=0}\tilde{g}=h.
\]
Using tildes to denote geometric quantities associated to $\tilde{g}$ and
undecorated characters to denote quantities associated to $g$, one recalls
from \cite[Lemma~2.2]{GIK02} that%
\begin{align*}
\left.  \frac{\partial}{\partial\varepsilon}\right\vert _{\varepsilon=0}%
\tilde{R}  &  =-(\Delta H-\delta^{2}h+\left\langle \operatorname*{Rc}%
,h\right\rangle ),\\
\left.  \frac{\partial}{\partial\varepsilon}\right\vert _{\varepsilon
=0}\mathrm{d\tilde{\mu}}  &  =\frac{1}{2}H\,\mathrm{d\mu}\,,\\
\left.  \frac{\partial}{\partial\varepsilon}\right\vert _{\varepsilon=0}%
\oint_{\mathcal{B}}\tilde{R}\,\mathrm{d\tilde{\mu}}\,  &  =\oint\left\{
\frac{1}{2}\left(  R-\oint R\,\mathrm{d\mu}\,\right)  H-\left\langle
\operatorname*{Rc},h\right\rangle \right\}  \,\mathrm{d\mu}\,.
\end{align*}
These formulas simplify nicely when $\operatorname*{Rc}=Kg$. Using this fact,
the result follows by careful but straightforward calculation.
\end{proof}

\subsection{Linear stability of volume-rescaled flow}

Now let us make the stronger assumption that $\mathbf{g}=(g,0,u)$ is a fixed
point of the volume-rescaled $\mathbb{R}^{1}$-invariant Ricci--DeTurck flow
with $u$ constant and $g$ a metric of constant sectional curvature $k>0$ and
hence $K=(n-1)k$.

By passing to a covering space if necessary, we may assume that $\mathcal{B}%
^{n}$ is the round $n$-sphere of radius $\sqrt{1/k}$. Recall that the spectrum
of the Laplacian $\Delta_{0}$ acting on scalar functions on $(\mathcal{B}%
^{n},g)$ is $\{\lambda_{j,k}\}_{j\geq0}$, where%
\[
\lambda_{j,k}=-jk(n+j-1).
\]
The eigenspace $\Lambda_{j,k}$ consists of the restriction to $\mathcal{B}%
^{n}\subset\mathbb{R}^{n+1}$ of all $j$-homogeneous polynomials in the
coordinate functions $(x^{1},\ldots,x^{n+1})$ of $\mathbb{R}^{n+1}$.

Given a symmetric $(2,0)$-tensor field $h$, let $f=f(h)$ denote its trace-free
part defined by
\begin{equation}
h=\frac{1}{n}Hg+f. \label{Define-f}%
\end{equation}
It is then convenient to rewrite (\ref{Volume-h}) and (\ref{Volume-B}) as
\begin{subequations}
\label{Einstein-NRF-LinearParabolic}%
\begin{align}
\mathbf{L}_{2}h  &  =\Delta h-2kf+2\frac{n-1}{n}k(H-\bar{H})g,\\
\mathbf{L}_{1}B  &  =\Delta_{1}B+(n-1)kB,
\end{align}
respectively.\footnote{Recall that $\Delta$ is the rough Laplacian.}

As above, the fact that $\mathbf{L}$ is diagonal lets us determine its
stability by examining its component operators. The conclusions we obtain will
continue to hold when we extend $\mathbf{L}$ to a larger domain in which
smooth representatives are dense.
\end{subequations}
\begin{lemma}
Let $\mathbf{g}=(g,0,u)$ be a metric of the form (\ref{KaluzaKleinMetric})
such that $u$ is constant and $g$ has constant sectional curvature $k>0$. Then
the linear system (\ref{volume-linear-parabolic}) has the following stability properties:

$\mathbf{L}_{0}$ is weakly stable; its null eigenspace consists of constant functions.

$\mathbf{L}_{1}$ is strictly stable.

If $n=2$, then $\mathbf{L}_{2}$ is weakly stable; its null eigenspace is
\[
\{\varphi g:\varphi\in\Lambda_{0,k}\cup\Lambda_{1,k}).
\]

If $n\geq3$, then $\mathbf{L}_{2}$ is unstable. The sole unstable eigenvalue
is $(n-2)k$ with eigenspace $\{\varphi g:\varphi\in\Lambda_{1,k}\}$. The null
eigenspace is $\{\varphi g:\varphi\in\Lambda_{0,k}\}=\{cg:c\in\mathbb{R}\}$.
\end{lemma}

\begin{proof}
Weak stability of $\mathbf{L}_{0}$ is clear.

It is well known (see \cite{GM75}, for example) that the spectrum of
$\Delta_{1}$ is%
\[
\left\{  \lambda_{j,k}\right\}  _{j\geq1}\cup\left\{  \lambda_{j,k}%
-(n-2)k\right\}  _{j\geq1}.
\]
Because $n\geq2$, it follows that the largest eigenvalue of $\mathbf{L}_{1}$
is $-k$.

Using (\ref{Define-f}), one obtains the decomposition $\mathbf{L}_{2}%
h=\frac{1}{n}L_{0}^{\prime}H+L_{2}^{\prime}f$, where%
\begin{align*}
L_{0}^{\prime}H  &  =\Delta_{0}H+2(n-1)k(H-\bar{H}),\\
L_{2}^{\prime}f  &  =\Delta f-2kf.
\end{align*}
The operator $L_{2}^{\prime}$ is evidently strictly stable. Because
\[
(L_{0}^{\prime}H,H)=-\left\Vert \nabla H\right\Vert ^{2}+2(n-1)k(\left\Vert
H\right\Vert ^{2}-V\bar{H}^{2}),
\]
where $V=\operatorname*{Vol}(\mathcal{B},g)$, one sees that the spectrum of
$L_{0}^{\prime}$ is bounded from above by that of $\Delta_{0}+2(n-1)k$.
Because $\lambda_{2,k}=-2(n+1)k$, the only possible unstable eigenfunctions
for $L_{0}^{\prime}$ are elements of $\Lambda_{0,k}$ and $\Lambda_{1,k}$. It
is easy to see that elements of $\Lambda_{0,k}$ (i.e.~constant functions)
belong to the nullspace of $L_{0}^{\prime}$. And if $H\in\Lambda_{1,k}$, then
one has $\bar{H}=0$ and hence $(L_{0}^{\prime}H,H)=\{\lambda_{1,k}%
+2(n-1)k\}\left\Vert H\right\Vert ^{2}=(n-2)k\left\Vert H\right\Vert ^{2}$.
Because $\left\Vert h\right\Vert ^{2}=\frac{1}{n}\left\Vert H\right\Vert
^{2}+\left\Vert f\right\Vert ^{2}$, the result follows.
\end{proof}

\begin{remark}
The elements of $\{\varphi g:\varphi\in\Lambda_{1,k}\}$ are infinitesimal
conformal diffeomorphisms (M\"{o}bius transformations): it is a standard fact
that any $h\in\{\varphi g:\varphi\in\Lambda_{1,k}\}$ satisfies
\[
h=\varphi g=-\frac{1}{2k}\mathcal{L}_{\operatorname*{grad}\varphi}g.
\]

\end{remark}

\begin{remark}
\label{UnstableHMF}The apparent instability of $\mathbf{L}_{2}$ for $n\geq3$
is an accident of the DeTurck trick rather than an essential feature of Ricci
flow. Indeed, Hamilton observed \cite{Ham95} that the DeTurck diffeomorphisms
solve a harmonic map heat flow, with an evolving domain metric and a fixed
target metric. It is well known that the identity map of round spheres
$\mathcal{S}^{n}\ $is an unstable harmonic map for all $n\geq3$. On the other
hand, one can show by different methods that spherical space forms are
attractive fixed points for normalized Ricci flow in any dimension. See
\cite{BW07} and \cite{BS08}.

In any case, the instability of $\mathbf{L}_{2}$ is in a sense a red herring,
i.e.~geometrically insignificant. Indeed, Moser \cite{Moser65} has shown that
if $g$ and $\tilde{g}$ are any metrics on a compact manifold $\mathcal{N}$
such that $\operatorname*{Vol}(\mathcal{N},\tilde{g})=\operatorname*{Vol}%
(\mathcal{N},g)$, then $\tilde{g}=\psi^{\ast}\hat{g}$ for some diffeomorphism
$\psi:\mathcal{N}\rightarrow\mathcal{N}$ and some metric $\hat{g}$ on
$\mathcal{N}$ with $\mathrm{d\tilde{\mu}}\,=\,\mathrm{d\mu}$. So, up to
diffeomorphism, it suffices to consider perturbations that preserve the volume
element $\mathrm{d\mu}$ of a stationary solution $g$. These are exactly those
$h$ with $H=0$ pointwise. Restricted to such perturbations, $\mathbf{L}_{2}$
is strictly stable.
\end{remark}

We can now at least partially explain our \emph{ad hoc} choices of normalization.

\begin{remark}
\label{kappa-bad}If $\mathbf{g}=(g,0,u)$ is a metric of the form
(\ref{KaluzaKleinMetric}) such that $u$ is constant and $g$ has constant
sectional curvature $1/2(n-1)>0$, then the operator $\mathbf{L}_{2}$ in the
linearization (\ref{kappa-linear-parabolic}) of $\varkappa$-rescaled
Ricci--DeTurck flow (\ref{kappa-rescaled}) satisfies%
\[
(\mathbf{L}_{2}h,h)=\frac{1}{n}(\left\Vert \nabla H\right\Vert ^{2}+\left\Vert
H\right\Vert ^{2})-\left\Vert \nabla f\right\Vert ^{2}-\frac{1}{n-1}\left\Vert
f\right\Vert ^{2}.
\]
So $1$ is an eigenvalue with eigenspace $\{cg:c\in\mathbb{R}\}$.
\end{remark}

\begin{remark}
\label{volume-bad}If $\mathbf{g}=(g,0,u)$ is a metric of the form
(\ref{KaluzaKleinMetric}) such that $u$ is constant and $g$ has constant
sectional curvature $k<0$, then Koiso's Bochner formula \cite{Koiso79} in the
form%
\[
\left\Vert \nabla h\right\Vert ^{2}=\frac{1}{2}\left\Vert T\right\Vert
^{2}+\left\Vert \delta h\right\Vert ^{2}-nk\left\Vert f\right\Vert ^{2}%
\]
implies that the operator $\mathbf{L}_{2}$ in the linearization
(\ref{volume-linear-parabolic}) of the volume-rescaled Ricci--DeTurck flow
(\ref{volume-rescaled}) satisfies%
\[
(\mathbf{L}_{2}h,h)=-\frac{1}{2}\left\Vert T\right\Vert ^{2}-\left\Vert \delta
h\right\Vert ^{2}+(n-2)k\left\Vert f\right\Vert ^{2}+2\frac{n-1}%
{n}k\{\left\Vert H\right\Vert ^{2}-V\bar{H}^{2}\}.
\]
H\"{o}lder's inequality implies that the quantity in braces is nonnegative and
vanishes exactly when $\left\vert H\right\vert $ is constant a.e. So
$\mathbf{L}_{2}$ has a null eigenvalue with eigenspace $\{cg:c\in\mathbb{R}\}$
in all dimensions. (The nullspace is of course larger when $n=2$).
\end{remark}

\subsection{Convergence and stability of volume-rescaled flow}

Here we exploit the fact that the identity map of the round $2$-sphere is a
weakly stable harmonic map.

\begin{proof}
[Proof of Theorem \ref{TwoSphereIsStable}]The proof is entirely analogous to
the proof of Theorem~\ref{HyperbolicIsStable}, except for Step~4. Here we must
uniformly bound diffeomorphisms $\{\varphi_{\tau}\}$ solving%
\begin{align*}
\frac{\partial}{\partial\tau}\varphi_{\tau}  &  =\Delta_{\widetilde
{\mathbf{g}}(\tau),\mathbf{g}}\varphi_{\tau},\\
\varphi_{0}  &  =\operatorname*{id}:(\mathcal{S}^{2},\widetilde{\mathbf{g}%
}(0))\rightarrow(\mathcal{S}^{2},\mathbf{g}),
\end{align*}
where $\mathbf{g}$ is a metric of constant positive curvature and
$\widetilde{\mathbf{g}}(\tau)\in B(\mathbb{X}_{1},\mathbf{g},\delta)$ for all
$\tau\geq0$. In this case, the required bound $\left\Vert \varphi_{\tau
}-\operatorname*{id}\right\Vert _{\mathbb{X}_{1}}\leq c(\delta)$ follows from
a result of\ Topping \cite{Topping97}. The remainder of the proof goes through
without modification.
\end{proof}

\section{The case $n=1$:\ holonomy-rescaled locally $\mathbb{R}^{N}$-invariant
Ricci flow\label{MTRF}}

We now consider evolving metrics on the mapping torus $\mathbb{R}%
^{N}\hookrightarrow\mathcal{M}_{\Lambda}\overset{\pi}{\mathcal{\longrightarrow
}}\mathcal{S}^{1}$ of a given $\Lambda\in\operatorname*{Gl}(N,\mathbb{R})$. We
let $x^{0}$ denote the coordinate on $\mathcal{S}^{1}\approx\mathbb{R}%
/\mathbb{Z}$ and take $(x^{1},\ldots,x^{N})$ as local coordinates on the
fibers. Let $c$ and $s$ be constants and let $V$ be a $1$-form to be
determined. Then the evolution of%
\begin{equation}
\mathbf{g}(\tau)=u\,dx^{0}\otimes dx^{0}+A_{i}\,(dx^{0}\otimes dx^{i}%
+dx^{i}\otimes dx^{0})+G_{ij}\,dx^{i}\otimes dx^{j} \label{CircleBundle}%
\end{equation}
by rescaled locally $\mathbb{R}^{N}$-invariant Ricci flow (\ref{RRFS})
modified by diffeomorphisms generated by the vector field $V^{\sharp}$ is
equivalent to the system
\begin{subequations}
\label{MappingTorus}%
\begin{align}
\frac{\partial}{\partial\tau}u  &  =\frac{1}{2}\left\vert \partial
_{0}G\right\vert ^{2}-su+(\mathcal{L}_{V^{\sharp}}\mathbf{g})_{00},\\
\frac{\partial}{\partial\tau}A^{i}  &  =-\frac{1+c}{2}sA^{i}+G^{ij}%
(\mathcal{L}_{V^{\sharp}}\mathbf{g})_{j0},\\
\frac{\partial}{\partial\tau}G_{ij}  &  =u^{-1}\{\nabla_{0}^{2}G_{ij}%
-(\partial_{0}G_{ik})G^{k\ell}(\partial_{0}G_{j\ell})\}+csG_{ij}%
+(\mathcal{L}_{V^{\sharp}}\mathbf{g})_{ij}.
\end{align}
We define%
\end{subequations}
\begin{equation}
V=\frac{C}{2}(\partial_{0}u)\,dx^{0}+(G_{ij}\partial_{0}A^{j})\,dx^{i},
\label{DefineVforMT}%
\end{equation}
where $C>0$ is a constant to be chosen below. Notice that $V$ implements a
DeTurck trick, as was done in Sections~\ref{KappaFlow}--\ref{VolumeFlow}.

Hereafter, we drop the boldface font and write $g$ for $\mathbf{g}$. For later
use, we note that the Christoffel symbols of $g$ are
\begin{subequations}
\label{MTconnection}%
\begin{align}
\Gamma_{00}^{0}  &  =\frac{1}{2}g^{00}\partial_{0}u+g^{0i}\partial_{0}%
(G_{ij}A^{j}),\\
\Gamma_{00}^{k}  &  =\frac{1}{2}g^{k0}\partial_{0}u+g^{ki}\partial_{0}%
(G_{ij}A^{j}),\\
\Gamma_{i0}^{0}  &  =\frac{1}{2}g^{0j}\partial_{0}G_{ij},\\
\Gamma_{ij}^{0}  &  =-\frac{1}{2}g^{00}\partial_{0}G_{ij},\\
\Gamma_{0j}^{k}  &  =\frac{1}{2}g^{k\ell}\partial_{0}G_{j\ell},\\
\Gamma_{ij}^{k}  &  =-\frac{1}{2}g^{k0}\partial_{0}G_{ij}.
\end{align}

\subsection{Holonomy-rescaled flow}

Following Lott \cite{Lott07a}, we say $G$ is a \emph{harmonic-Einstein metric}
if%
\end{subequations}
\begin{equation}
\nabla_{0}^{2}G_{ij}=(\partial_{0}G_{ik})G^{k\ell}(\partial_{0}G_{\ell j}).
\label{Harmonic-Einstein}%
\end{equation}

We henceforth assume that $\mathcal{M}_{\Lambda}$ admits a metric $g$ of the
form (\ref{CircleBundle}) with $G$ harmonic-Einstein. By an initial
reparameterization, we may assume without loss of generality that $u=1$ at
$t=0$. As we shall observe in Lemma~\ref{StationaryMT} below, the
harmonic-Einstein condition implies that $\left\vert \partial_{0}G\right\vert
^{2}$ is constant in space. Motivated by Example~\ref{Sol}, we choose
constants $c=0$ and $s=\frac{1}{2}\left\vert \partial_{0}G\right\vert ^{2}$.
Observe that the constant $s$ is determined by the holonomy of $\mathcal{M}%
_{\Lambda}$, because for any fiber metric $G$, one has%
\begin{equation}
G_{ij}|_{x^{0}=1}=\Lambda_{i}^{k}\Lambda_{j}^{\ell}G_{k\ell}|_{x^{0}=1}.
\label{Holonomy}%
\end{equation}
So when $c$, $s$, and $V$ are chosen in this way, we call system
(\ref{MappingTorus}) \emph{holonomy-rescaled locally }$\mathbb{R}^{N}%
$\emph{-invariant Ricci flow} and denote a solution by $g(\tau)=(u(\tau
),A(\tau),G(\tau))$.

\begin{lemma}
\label{StationaryMT}Let $g$ be a metric of the form (\ref{CircleBundle}) such
that $u=1$, $A$ vanishes, and $G$ is a harmonic-Einstein metric. Then $g$ is a
stationary solution of (\ref{MappingTorus}).
\end{lemma}

\begin{proof}
The hypotheses on $u$ and $A$ imply that $V=0$ at $\tau=0$. The fact that $G$
is harmonic-Einstein implies that $\left\vert \partial_{0}G\right\vert ^{2}$
is constant in space. Hence $\frac{1}{2}\left\vert \partial_{0}G\right\vert
^{2}=su$. The result follows.
\end{proof}

Notice that the $\operatorname*{sol}$-geometry manifolds in Example~\ref{Sol}
satisfy the hypotheses of the lemma. There $s$ is the topological constant
$s=4(\log\lambda)^{2}$. Clearly, $s>0$ whenever $\mathcal{M}_{\Lambda}$ has
nontrivial holonomy.

\subsection{Linearization at a stationary solution of holonomy-rescaled flow}

Let $(u+v,B,G+h)$ be a perturbation of a stationary solution $g=(1,0,G)$ of
the type considered in Lemma~\ref{StationaryMT}.

\begin{lemma}
The linearization of holonomy-rescaled locally $\mathbb{R}^{N}$-invariant
Ricci flow (\ref{MappingTorus}) about a stationary solution $(1,0,G)$ with $G$
harmonic-Einstein$\ $is the autonomous, strictly parabolic system%
\[
\frac{\partial}{\partial\tau}%
\begin{pmatrix}
v & B & h
\end{pmatrix}
=\mathbf{L}%
\begin{pmatrix}
v & B & h
\end{pmatrix}
=%
\begin{pmatrix}
\mathbf{L}_{0}v+\mathbf{F}h & \mathbf{L}_{1}B & \mathbf{L}_{2}h
\end{pmatrix}
,
\]
where
\begin{subequations}
\label{LinearMT}%
\begin{align}
\mathbf{L}_{0}v  &  =C\partial_{0}^{2}v-sv,\label{LinearMT-v}\\
\mathbf{F}h  &  =\left\langle \partial_{0}G,\partial_{0}h\right\rangle
-\left\langle \partial_{0}^{2}G,h\right\rangle ,\\
(\mathbf{L}_{1}B)^{i}  &  =\partial_{0}^{2}B^{i}-\frac{s}{2}B^{i}%
,\label{LinearMT-B}\\
(\mathbf{L}_{2}h)_{ij}  &  =\partial_{0}^{2}h_{ij}-G^{k\ell}(\partial
_{0}h_{ik}\partial_{0}G_{\ell j}+\partial_{0}G_{ik}\partial_{0}h_{\ell
j})+\partial_{0}G_{ik}h^{k\ell}\partial_{0}G_{\ell j}. \label{LinearMT-h}%
\end{align}

\end{subequations}
\end{lemma}

\begin{proof}
One may suppose that $\left.  \frac{\partial}{\partial\varepsilon}\right\vert
_{\varepsilon=0}u=v$, $\left.  \frac{\partial}{\partial\varepsilon}\right\vert
_{\varepsilon=0}A=B$, and $\left.  \frac{\partial}{\partial\varepsilon
}\right\vert _{\varepsilon=0}G=h$. We first compute that%
\[
\left.  \frac{\partial}{\partial\varepsilon}\right\vert _{\varepsilon
=0}V=\frac{C}{2}\partial_{0}v\,dx^{0}+G_{ij}\partial_{0}B^{j}\,dx^{i}%
\]
and%
\begin{align*}
\left.  \frac{\partial}{\partial\varepsilon}\right\vert _{\varepsilon
=0}(\mathcal{L}_{V}g)_{00}  &  =C\partial_{0}^{2}v,\\
\left.  \frac{\partial}{\partial\varepsilon}\right\vert _{\varepsilon
=0}(\mathcal{L}_{V}g)_{j0}  &  =G_{ij}\partial_{0}^{2}B^{i},\\
\left.  \frac{\partial}{\partial\varepsilon}\right\vert _{\varepsilon
=0}(\mathcal{L}_{V}g)_{ij}  &  =\frac{1}{2}(\partial_{0}v)(\partial_{0}%
G_{ij}).
\end{align*}
Note that we used (\ref{MTconnection}) to get the last equality. Using the
fact that $G$ is harmonic-Einstein, we obtain%
\[
\left.  \frac{\partial}{\partial\varepsilon}\right\vert _{\varepsilon
=0}\left\vert \partial_{0}G\right\vert ^{2}=2\left\langle \partial
_{0}G,\partial_{0}h\right\rangle -2\left\langle \partial_{0}^{2}%
G,h\right\rangle ,
\]
where the inner products are taken with respect to $G$. Observing that%
\[
\left.  \frac{\partial}{\partial\varepsilon}\right\vert _{\varepsilon=0}%
\nabla_{0}^{2}G_{ij}=\partial_{0}^{2}h_{ij}-(\partial_{0}G_{ij})\left(
\left.  \frac{\partial}{\partial\varepsilon}\right\vert _{\varepsilon=0}%
\Gamma_{00}^{0}\right)  =\partial_{0}^{2}h_{ij}-\frac{1}{2}(\partial
_{0}v)(\partial_{0}G_{ij}),
\]
we conclude that%
\begin{multline*}
\left.  \frac{\partial}{\partial\varepsilon}\right\vert _{\varepsilon
=0}\left\{  u^{-1}[\nabla_{0}^{2}G_{ij}-(\partial_{0}G_{ik})G^{k\ell}%
(\partial_{0}G_{j\ell})]\right\} \\
=\partial_{0}^{2}h_{ij}-G^{k\ell}(\partial_{0}h_{ik}\partial_{0}G_{\ell
j}+\partial_{0}G_{ik}\partial_{0}h_{\ell j})+\partial_{0}G_{ik}h^{k\ell
}\partial_{0}G_{\ell j}-\frac{1}{2}(\partial_{0}v)(\partial_{0}G_{ij}).
\end{multline*}
The result follows.
\end{proof}

\subsection{Linear stability of holonomy-rescaled flow}

To investigate linear stability of (\ref{LinearMT}), it is convenient to `cut'
the bundle $\mathcal{M}_{\Lambda}$ and consider tensor fields depending on
$x^{0}\in\lbrack0,1]$. By our assumption that $\mathcal{M}_{\Lambda}$ admits a
flat connection, we may henceforth assume that the $\mathbb{R}^{N}$-valued
$1$-form $A$ has trivial holonomy, i.e.~that $A^{i}|_{x^{0}=1}=A^{i}%
|_{x^{0}=0}$, while $G$ satisfies (\ref{Holonomy}) for a nontrivial action of
$\Lambda\in\operatorname*{Gl}(N,\mathbb{R})$. Notice that these assumptions
are satisfied by $\operatorname*{sol}^{3}$-twisted torus bundles, of which
Example~\ref{Sol} is a special case.

Corresponding to these assumptions, we require $v$ and $B$ to satisfy the
boundary conditions%
\begin{align}
v|_{x^{0}=1}  &  =v|_{x^{0}=0},\label{MTboundary-v}\\
B^{i}|_{x^{0}=1}  &  =B^{i}|_{x^{0}=0}, \label{MTboundary-B}%
\end{align}
respectively. We require $h$ to satisfy the linear compatibility condition
\begin{equation}
h_{ij}|_{x^{0}=1}=\Lambda_{i}^{k}\Lambda_{j}^{\ell}h_{k\ell}|_{x^{0}=0}
\label{MTboundary-h}%
\end{equation}
derived from (\ref{Holonomy}).

For now, we apply the differential operators in (\ref{LinearMT}) only to
smooth data satisfying these boundary conditions. Below, we will specify
larger domains in which smooth representatives are dense. With this intention,
we now investigate the linear stability of (\ref{LinearMT}). We will
eventually have to consider the entire action of the lower-triangular operator
$\mathbf{L}$, but we begin with observations about its component operators.
The first two are self evident.

\begin{lemma}
$\mathbf{L}_{0}$ is self adjoint with respect to the inner product%
\[
\left(  \varphi,\psi\right)  =\int_{0}^{1}\varphi\psi\,d\xi
\]
and is strictly linearly stable if $s>0$.
\end{lemma}

\begin{lemma}
$\mathbf{L}_{1}$ is self adjoint with respect to the inner product%
\[
\left(  \varphi,\psi\right)  =\int_{0}^{1}\varphi^{i}\psi^{j}\delta_{ij}\,d\xi
\]
and is strictly linearly stable if $s>0$.
\end{lemma}

\begin{lemma}
$\mathbf{L}_{2}$ is self adjoint with respect to the inner product%
\[
\left(  \varphi,\psi\right)  =\int_{0}^{1}\left\langle \varphi,\psi
\right\rangle \,d\xi=\int_{0}^{1}\varphi_{ij}\psi_{k\ell}G^{ik}G^{j\ell}\,d\xi
\]
and$\mathbf{\ }$is weakly linearly stable. Its null eigenspace is%
\[
\{GM:M\in\operatorname*{gl}(N,\mathbb{R})\text{ commutes with all }%
G(\cdot)\}.
\]

\end{lemma}

\begin{proof}
If $\varphi$ and $\psi$ satisfy (\ref{MTboundary-h}) and are smooth, then
$\partial_{0}\varphi_{ij}|_{x^{0}=1}=\Lambda_{i}^{k}\Lambda_{j}^{\ell}%
\partial_{0}\varphi_{k\ell}|_{x^{0}=0}$ and $\partial_{0}\psi_{ij}|_{x^{0}%
=1}=\Lambda_{i}^{k}\Lambda_{j}^{\ell}\partial_{0}\psi_{k\ell}|_{x^{0}=0}$. So
if $G$ satisfies (\ref{Holonomy}), then all boundary terms arising from
integrations by parts vanish. Using this fact, it is easy to verify that
\[
\left(  \mathbf{L}_{2}\varphi,\psi\right)  =\left(  \varphi,\mathbf{L}_{2}%
\psi\right)  .
\]

Now we write $\left(  \mathbf{L}_{2}h,h\right)  =I_{1}-2I_{2}+I_{3}$, where
(using primes to denote differentiation with respect to $x^{0}=\xi$ inside an
integral) the distinct terms are%
\begin{align*}
I_{1}  &  =\int_{0}^{1}h_{ij}^{\prime\prime}h_{k\ell}G^{ik}G^{j\ell}\,d\xi,\\
I_{2}  &  =\int_{0}^{1}h_{ip}^{\prime}G^{pq}G_{qj}^{\prime}h_{k\ell}%
G^{ik}G^{jl}\,d\xi,\\
I_{3}  &  =\int_{0}^{1}G_{ip}^{\prime}G^{pq}h_{qj}G_{kr}^{\prime}%
G^{rs}h_{s\ell}G^{i\ell}G^{jk}\,d\xi.
\end{align*}
An integration by parts shows that%
\[
I_{1}=-\left\Vert \partial_{0}h\right\Vert ^{2}+2I_{2}.
\]
The integral $I_{3}$, which may be rewritten as%
\[
I_{3}=\left(  \partial_{0}G\circ h,h\circ\partial_{0}G\right)  ,
\]
requires more work. Integrating it by parts produces seven terms:%
\[
I_{3}=2I_{3}+2\left\Vert \partial_{0}G\circ h\right\Vert ^{2}-2\left(
\partial_{0}G\circ h,\partial_{0}h\right)  -\left(  \partial_{0}^{2}G\circ
h,h\right)  ,
\]
Because $G$ satisfies the harmonic-Einstein equation (\ref{Harmonic-Einstein}%
), the final term above may be rewritten as%
\[
\left(  \partial_{0}^{2}G\circ h,h\right)  =\left\Vert \partial_{0}G\circ
h\right\Vert ^{2}.
\]
Collecting terms and applying Cauchy--Schwarz, we obtain%
\begin{align*}
\left(  \mathbf{L}_{2}h,h\right)   &  =-\left\Vert \partial_{0}h\right\Vert
^{2}-\left\Vert \partial_{0}G\circ h\right\Vert ^{2}+2\left(  \partial
_{0}G\circ h,\partial_{0}h\right) \\
&  \leq-\left\Vert \partial_{0}h\right\Vert ^{2}-\left\Vert \partial_{0}G\circ
h\right\Vert ^{2}+2\left\Vert \partial_{0}G\circ h\right\Vert \left\Vert
\partial_{0}h\right\Vert \leq0.
\end{align*}
One has equality if and only if $h_{j}^{k}=G^{k\ell}h_{\ell j}$ is independent
of $x^{0}$, hence if and only if $h_{ij}=G_{ik}M_{j}^{k}$ for some
$M\in\operatorname*{gl}(N,\mathbb{R})$ that commutes with all $G(x^{0})$.
\end{proof}

\begin{remark}
An element $GM$ of the nullspace of $\mathbf{L}_{2}$ will itself satisfy the
harmonic-Einstein equation (\ref{Harmonic-Einstein}) if it is a metric,
i.e.~if $M$ is invertible.
\end{remark}

\begin{remark}
For the $\operatorname*{sol}$-Gowdy metrics considered in Example~\ref{Sol},
it is easy to verify that the nullspace of $\mathbf{L}_{2}$ is two-dimensional.
\end{remark}

Now we extend the $L^{2}$ inner products above to vectors $%
\begin{pmatrix}
v & B & h
\end{pmatrix}
$ in the natural way, i.e.%
\[
\left(
\begin{pmatrix}
v & B & h
\end{pmatrix}
,%
\begin{pmatrix}
\alpha & \beta & \gamma
\end{pmatrix}
\right)  =\left(  v,\alpha\right)  +\left(  B,\beta\right)  +\left(
h,\gamma\right)  .
\]
The lemma below is the main observation of this subsection. Observe that it
does not directly imply convergence of the nonlinear system
(\ref{MappingTorus}), because $\mathbf{L}$ is not self adjoint. In the next
subsection, after introducing suitable function spaces, we shall overcome this
difficulty by exhibiting $\mathbf{L}$ as a sectorial perturbation of the
self-adjoint operator%
\begin{equation}
\mathbf{L}_{\mathrm{sa}}:%
\begin{pmatrix}
v & B & h
\end{pmatrix}
\mapsto%
\begin{pmatrix}
\mathbf{L}_{0}v & \mathbf{L}_{1}B & \mathbf{L}_{2}h
\end{pmatrix}
. \label{LSA}%
\end{equation}

\begin{lemma}
Assume that $G$ has nontrivial holonomy $\Lambda\in\operatorname*{Gl}%
(N,\mathbb{R})$ but admits a flat connection. Then there exists $C>0$
depending only on $G$ such that the linearization $\mathbf{L}$ of
holonomy-rescaled locally $\mathbb{R}^{N}$-invariant Ricci flow is weakly
linearly stable with respect to the DeTurck diffeomorphisms induced by
$V=V(C)$ defined in (\ref{DefineVforMT}). Its nullspace is%
\[
\{GM:M\in\operatorname*{gl}(N,\mathbb{R})\text{ commutes with all }%
G(\cdot)\}.
\]

\end{lemma}

\begin{proof}
First suppose that $h$ belongs to the null eigenspace of $\mathbf{L}_{2}$.
Then because $G$ is harmonic-Einstein and $h=GM$, one has $\left\langle
\partial_{0}^{2}G,h\right\rangle =\left\langle \partial_{0}G,\partial
_{0}h\right\rangle $. Hence $\mathbf{F}h=0$ and so%
\[
\left(  \mathbf{L}%
\begin{pmatrix}
v & B & h
\end{pmatrix}
,%
\begin{pmatrix}
v & B & h
\end{pmatrix}
\right)  =-C\left\Vert \partial_{0}v\right\Vert ^{2}-s\left\Vert v\right\Vert
^{2}-\left\Vert \partial_{0}B\right\Vert ^{2}-\frac{s}{2}\left\Vert
B\right\Vert ^{2}.
\]
This is strictly negative unless $v$ and $B$ vanish identically.

For the general case, write $h=h^{\circ}+h^{\perp}$, where $h^{\circ}$ belongs
to the nullspace of $\mathbf{L}_{2}$ and $h^{\perp}$ belongs to its orthogonal
complement. The self-adjoint elliptic operator $\mathbf{L}_{2}$ has pure point
spectrum with eigenvalues of finite multiplicity. So there exists $\lambda>0$
depending only on $G$ such that $\left(  \mathbf{L}_{2}h^{\perp},h^{\perp
}\right)  \leq-\lambda\left\Vert h^{\perp}\right\Vert ^{2}$. Hence%
\begin{multline*}
\left(  \mathbf{L}%
\begin{pmatrix}
v & B & h
\end{pmatrix}
,%
\begin{pmatrix}
v & B & h
\end{pmatrix}
\right)  =-C\left\Vert \partial_{0}v\right\Vert ^{2}-s\left\Vert v\right\Vert
^{2}+\left(  Fh^{\perp},v\right) \\
-\left\Vert \partial_{0}B\right\Vert ^{2}-\frac{s}{2}\left\Vert B\right\Vert
^{2}-\lambda\left\Vert h^{\perp}\right\Vert ^{2}.
\end{multline*}
Now for \emph{any }$h$, one integrates by parts and uses the harmonic-Einstein
identity (\ref{Harmonic-Einstein}) to see that%
\[
\left(  Fh,v\right)  =-\int_{0}^{1}\left\langle G^{\prime},h\right\rangle
v^{\prime}\,d\xi.
\]
There exists $\kappa$ depending only on $G$ such that $\left\vert \left\langle
G^{\prime},h\right\rangle \right\vert \leq\kappa$ $\left\vert h\right\vert $
pointwise. Thus one estimates that%
\begin{multline*}
\left(  \mathbf{L}%
\begin{pmatrix}
v & B & h
\end{pmatrix}
,%
\begin{pmatrix}
v & B & h
\end{pmatrix}
\right)  \leq-s\left\Vert v\right\Vert ^{2}-\frac{s}{2}\left\Vert B\right\Vert
^{2}-\frac{\lambda}{2}\left\Vert h^{\perp}\right\Vert ^{2}\\
-\left\{  C\left\Vert \partial_{0}v\right\Vert ^{2}-\kappa\left\Vert
\partial_{0}v\right\Vert \left\Vert h^{\perp}\right\Vert +\frac{\lambda}%
{2}\left\Vert h^{\perp}\right\Vert ^{2}\right\}  .
\end{multline*}
The quantity in braces is nonnegative as long as $C\geq\kappa^{2}/(2\lambda)$.
The result follows.
\end{proof}

\subsection{Convergence and stability of holonomy-rescaled flow}

\begin{proof}
[Proof of Theorem \ref{SolIsStable}]We again follow the proof of
Theorem~\ref{HyperbolicIsStable}, \emph{mutatis mutandis, }except for two
critical steps.

\textsc{Step 1.} The complexification $\mathbf{L}_{\mathrm{sa}}^{\mathbb{C}}$
of the operator defined in (\ref{LSA}) is strictly elliptic and self adjoint
with bounded spectrum, hence is sectorial by standard Schauder estimates.
Observe that $\mathbf{L}^{\mathbb{C}}-\mathbf{L}_{\mathrm{sa}}^{\mathbb{C}%
}=\mathbf{F}^{\mathbb{C}}$, where $\mathbf{F}^{\mathbb{C}}\in\mathcal{L}%
(\Sigma_{2}^{1+\alpha},\Sigma_{2}^{0+\alpha})$ for all $\alpha\in(0,1)$.
Namely, $\mathbf{F}^{\mathbb{C}}$ is a bounded operator from intermediate
spaces $\mathbb{X}_{\vartheta}=(\mathbb{X}_{0},\mathbb{X}_{1})_{\vartheta}$
and $\mathbb{E}_{\vartheta}=(\mathbb{E}_{0},\mathbb{E}_{1})_{\vartheta}$ to
$\mathbb{X}_{0}$ and $\mathbb{E}_{0}$, respectively, with its bound depending
only on the stationary solution about which we are linearizing. The fact that
$\mathbf{L}^{\mathbb{C}}$ is sectorial then follows from classical
perturbation results. For instance, see \cite[Proposition~2.4.1]{Lunardi95}.

\textsc{Step 4.} The only change in this step, once again, is how one controls
the diffeomorphism $\varphi_{\tau}$. Here the argument is much easier, because
the base $\mathcal{S}^{1}$ is flat. By (\ref{DefineVforMT}), the vector fields
$V$ that generate the diffeomorphisms $\psi_{\tau}$ satisfy%
\[
\left\Vert V\right\Vert _{1+\rho}\leq C_{N}\left\Vert \pi^{\mathrm{s}}%
\widehat{\mathbf{g}}(\tau)-\gamma(\pi^{\mathrm{c}}\widehat{\mathbf{g}}%
(\tau))\right\Vert _{\mathbb{X}_{1}}\leq C_{N}Ce^{-\lambda\tau}\left\Vert
\widetilde{\mathbf{g}}(0)-\mathbf{g}\right\Vert _{\mathbb{X}_{\alpha}}.
\]
It is then not hard to see that the diffeomorphisms $\psi_{\tau}$
and$\ \varphi_{\tau}$ must converge. Indeed, this is proved in
\cite[Lemma~3.5]{GIK02}. The remainder of the argument goes through without modification.
\end{proof}

\end{document}